\documentclass[12pt]{amsart}

\calclayout

\usepackage[dvipsnames]{xcolor}
\usepackage{caption}            % figures
\usepackage{subcaption}         % subfigures
\usepackage{multirow}           % multirow/col tables
\usepackage{listings}           % code 
\usepackage{url}      
\usepackage[lite]{amsrefs}
\usepackage{amssymb}
\usepackage{float}    

\usepackage[all,cmtip]{xy}

\usepackage{mathrsfs}
\usepackage{tikz-cd}
\usepackage{tabularx}
\usepackage{booktabs}
\usepackage[labelfont=bf,format=plain,justification=raggedright,singlelinecheck=false]{caption}

\newcommand{\F}{\mathbb{F}}

\newcommand{\pc}{\mathcal{P}}
\newcommand{\qc}{\mathcal{Q}}
\numberwithin{equation}{section}
\newcommand{\m}[1]{\ (\text{mod } #1)}
\definecolor{airforceblue}{rgb}{0.36, 0.54, 0.66}
\definecolor{amber1}{rgb}{1.0, 0.49, 0.0}
\definecolor{amber}{rgb}{1.0, 0.75, 0.0}
\definecolor{antiquefuchsia}{rgb}{0.57, 0.36, 0.51}

\theoremstyle{plain}
\newtheorem{theorem}[equation]{Theorem}
\newtheorem{corollary}[equation]{Corollary}
\newtheorem{lemma}[equation]{Lemma}

\theoremstyle{definition}

\theoremstyle{remark}
\newtheorem{remark}[equation]{Remark}

\usepackage{bookmark}

\usepackage{hyperref}

\newcommand{\fl}[1]{\left\lfloor #1 \right\rfloor}

\begin{document}

\title[Fermat curve]{On Weierstrass semigroups of maximal Fermat function fields}

\author{Peter Beelen} \address{Department of Applied Mathematics and Computer Science, Technical University of Denmark, Kongens Lyngby 2800, Denmark} \email{pabe@dtu.dk} \thanks{}
\author{Maria Montanucci} \address{Department of Applied Mathematics and Computer Science, Technical University of Denmark, Kongens Lyngby 2800, Denmark} \email{marimo@dtu.dk} \thanks{}
\author{Marie Frank vom Braucke} \address{Department of Applied Mathematics and Computer Science, Technical University of Denmark, Kongens Lyngby 2800, Denmark}  \email{mfvbr@dtu.dk} \thanks{}

\address{}

\date{}

 \begin{abstract}
In this article we explicitly determine the Weierstrass semigroup at any place of some $\mathbb{F}_{q^2}$-maximal Fermat function fields $\mathcal{F}_m$, namely for $m=(q+1)/2$ and $m=(q+1)/3$. These famous function fields arise as Galois subfields of the Hermitian function field, and even though they have been intensively studied in the literature, the Weierstrass semigroup at every place is still not fully known.
Surprisingly enough this problem is in fact quite involved and $\mathcal{F}_m$ has many different types of Weierstrass semigroups. Moreover, its set of Weierstrass places is much richer than its set of rational places.
\end{abstract}

\maketitle

\vspace{0.5cm}\noindent {\em Keywords}: Maximal function field, Weierstrass semigroup, Weierstrass places

\vspace{0.2cm}\noindent{\em MSC}: Primary: 11G20. Secondary: 11R58, 14H05, 14H55

\vspace{0.2cm}\noindent

\section{Introduction}
\label{sec:introduction}

%\textcolor{red}{Literature on Fermat curves. Mention that HKT studies the curve for $m=(q+1)/2$ in Chapter 10, pages 436-439. Paper on Fermat curves: when maximal. m needs to divide q+1 Tafazolian, Garcia and Torres On maximal curves of Fermat curves. Arakelian and Borges. m=(q+1)/3. Third largest genus $X_3$. Codes and maximal curves of high genus. Garcia-Voloch??}

An algebraic function field $F$ over a finite field with square cardinality is called maximal if its number of places of degree one attains the Hasse-Weil upper bound. More precisely, if $F$ is an algebraic function field of genus $g$
defined over the finite field $\mathbb{F}_{q^2}$ of cardinality $q^2$, then the Hasse-Weil upper bound states that
$$N(F) \le q^2 + 1 + 2qg,$$
where $N(F)$ denotes the number of places of degree 1 (sometimes also called $\mathbb{F}_{q^2}$-rational places) of $F$.
Maximal function fields, especially those with large genus, have been and still are extensively investigated in the literature. Apart from their being extremal objects, and hence being theoretically interesting, they attracted particular attention in connection with coding theory and cryptography, due to Goppa's construction of error-correcting codes \cite{Goppa}. In this construction algebraic function fields with many places of degree one are particularly suitable to give rise to good, long error-correcting codes. Therefore, it is natural to study maximal function fields and particularly those of large genus.
In constructing error-correcting codes from algebraic function fields, Weiestrass semigroups and Weierstrass places play a central role. This is why they have been intensively investigated in relation with maximal function fields, see for example \cites{BM,vicino,FT,GT,GTo,GV}. Given a place $P$ on $F$, the Weierstrass semigroup $H(P)$ is defined as the set of nonnegative
integers $n$ for which there exists a function $f \in F$ with $(f)_\infty = nP$. From the Weierstrass gap Theorem \cite[Theorem 1.6.8]{Sti}, the set $G(P) := \mathbb{N} \backslash H(P)$ contains exactly $g$ elements (called gaps). Clearly the set $H(P)$ in general may vary as the place $P$ varies, and different Weierstrass semigroups can be expected at different places of $F$. However, it is known that generically the semigroup $H(P)$ is the same, but that there can exist finitely many places of $F$, called Weierstrass places, with a different set of gaps.

The interest and potential of Weierstrass semigroups do not only live in their uses in coding theory, but in fact also theoretical applications are known. An interesting one is in relation with finding characterizing properties of algebraic function fields, a fact that was already pointed out in Stöhr-Voloch theory \cite{SV}. 

In this paper we analyze Weierstrass semigroups of certain Fermat function fields. Given $q$ a power of a prime $p$ and $m$ coprime with $p$, the Fermat function field $\mathcal{F}_m=\mathbb{F}_{q^2}(X,Y)$ is given by the defining equation 
$$X^m+Y^m+1=0.$$

The Fermat function fields are among some of the most studied algebraic function fields. Apart for their purely theoretical interest \cite{Borg, MoS, Sat, TT1} and their use in coding theory \cite{Sp}, they arise in finite geometry, where it was shown that non-Desarguesian finite flag-transitive projective planes exist if and only if certain Fermat function fields have no nontrivial rational places \cite{KZ}, in the construction of arcs in finite projective planes \cite{B, GPU}, and in the study of algebraically constructed lattices \cite{PS}, among others.

The function field $\mathcal{F}_m$ is known to be maximal over $\mathbb{F}_{q^2}$ precisely when $m$ divides $q+1$, see \cite{GT}. In this manuscript we compute the Weierstrass semigroup at every place of the function field $\mathcal{F}_m$ for $m=(q+1)/2$ and $m=(q+1)/3$.
This is the first time that a full investigation of this type is done for this class of function fields. Partial results for the first case can be found in \cite[Chapter 10, pages 436-439]{HKT}, while partial results for the latter are available in \cite{BH}.

The paper is organized as follows. In Section \ref{sec:preliminaries} we recall some preliminary notions and results about function fields and Weierstrass semigroups, together with a collection of some initial properties of the function field $\mathcal{F}_m$. In Sections \ref{sec:three} and \ref{sec_four} we compute the Weierstrass semigroups at every place of $\mathcal{F}_m$ for $m=(q+1)/2$ and $m=(q+1)/3$. Section \ref{sec:three} starts analyzing the Weierstrass semigroups at some special places of $\mathcal{F}_m$ for any $m$ dividing $q+1$, as well as considering some explicit power series expansion for function that will be used later to generate elements of the Weierstrass semigroups $H(P)$. Then the section specializes in different subsections on the two particular values of $m$ written above. As it becomes clear in Section \ref{sec:three}, the case $m=(q+1)/3$ is more involved and computational. In fact Section \ref{sec_four} is devoted to completing this case. 

\section{Preliminaries}
\label{sec:preliminaries}
In this section, we deal with the preliminary notions and results that will be needed throughout the paper. In the first subsection, we recall the definition of the Fermat function field $\mathcal{F}_m$ and we focus on some particular rational functions defined on it, computing their principal divisors. In the second subsection, we collect some preliminaries on numerical/Weierstrass semigroups and their connection with regular differentials. 

\subsection{The function field $\mathcal{F}_m$.}
Let $q$ be a prime power and let $m$ be an arbitrary divisor of $q+1$. Let $\mathbb{F}_{q^2}$ denote the finite field with $q^2$ elements, ${\bar{\mathbb{F}}}_{q^2}$ its algebraic closure and let $p$ be the characteristic of $\mathbb{F}_{q^2}$. The Fermat function field of degree $m$, $\mathcal{F}_m$, is defined as $\mathcal{F}_m:={\bar{\mathbb{F}}}_{q^2}(X,Y)$ with 
$$X^m+Y^m+1=0.$$
Note that we have used $\mathcal{F}_m$ with a slight abuse of notation. Typically in this setting, the constant field of the Fermat function field is taken to be $\mathbb{F}_{q^2}$, but since we aim to study the Weierstrass semigroup of all places, it is convenient for us to work over an algebraically closed constant field. Likewise with a slight abuse of notation, we will call a place of $\mathcal{F}_m$ rational, if it can be identified with a rational place of $\mathbb{F}_{q^2}(X,Y)$.

It is easy to see that both function field extensions $\mathcal{F}_m/\overline{\mathbb{F}}_{q^2}(X)$ and $\mathcal{F}_m/\overline{\mathbb{F}}_{q^2}(Y)$ are Kummer extensions of degree $m$ with exactly $m$ (totally) ramified places, namely the zeroes of $X^m+1$ and $Y^m+1$ respectively (that is $(X=\alpha_i)$ and $(Y=\alpha_i)$ with $\alpha_i^m+1=0$ for $i=1,\ldots,m$, respectively). For $\alpha_i \in \mathbb{F}_{q^2}$, $\alpha_i^m+1=0$, We denote with $P_{(0,\alpha_i)}$ (resp. $P_{(\alpha_i,0)}$) the place above $(Y=\alpha_i)$ in $\mathcal{F}_m/\overline{\mathbb{F}}_{q^2}(Y)$ (resp. $(X=\alpha_i)$ in $\mathcal{F}_m/\overline{\mathbb{F}}_{q^2}(X)$). In general we will denote with $P_{(a,b)}$, $a,b \in \overline{\mathbb{F}}_{q^2}$ the place above both $(X=a)$ and $(Y=b)$ in the two function field extensions mentioned above. Denoting finally with $P_{\infty}^i$ for $i=1,\ldots,m$ the common poles of $X$ and $Y$ one has
\begin{equation}
\label{div:x}
(X)_{\mathcal{F}_m}= \sum_{i=1}^{m} P_{(0,\alpha_i)}  - \sum_{i = 1}^{m}P_{\infty}^{i},
\end{equation}
\begin{equation}
\label{div:y}
(Y)_{\mathcal{F}_m}=\sum_{i=1}^{m} P_{(\alpha_i,0)}  - \sum_{i = 1}^{m}P_{\infty}^{i},
\end{equation}
\begin{equation}
\label{div:xalphai}
     (X-\alpha_i)_{\mathcal{F}_m}= mP_{(\alpha_i,0)}  - \sum_{j = 1}^{m}P_{\infty}^{j},
     \end{equation}
     and
\begin{equation*}
%\label{div:yalphai}
        (Y-\alpha_i)_{\mathcal{F}_m}= mP_{(0,\alpha_i)}  - \sum_{j = 1}^{m}P_{\infty}^{j}.
\end{equation*}
Furthermore \cite{Sti}*{Corollary 3.7.4} implies that $g(\mathcal{F}_m)=(m-1)(m-2)/2$.

More generally, let $\zeta_m\in\F_{q^2}$ be an $m$-th root of unity, then if we have $a,b \in \overline{\mathbb{F}}_{q^2} \setminus \{\alpha_i: \ i=1,\ldots,m\}$, and $a^m + b^m +1 = 0$, we also have the divisors
    \begin{align*}
        (X-a)_{\mathcal{F}_m} & = \sum_{i = 1}^{m} P_{(a,\zeta_m^i b)} - \sum_{i = 1}^{m}P_{\infty}^{i}, \text{ and}\\
        (Y-b)_{\mathcal{F}_m} & = \sum_{i = 1}^{m} P_{(\zeta_m^i a, b)} - \sum_{i = 1}^{m}P_{\infty}^{i}. 
    \end{align*}
The function field $\mathcal{F}_m$, seen as a function field over $\mathbb{F}_{q^2}$ is $\mathbb{F}_{q^2}$-maximal by Serre's covering result \cite{KS}. In fact, it can be seen as a subfield of the Hermitian function field $\mathcal{F}_{q+1}:=\mathbb{F}_{q^2}(U,V)$ with $U^{q+1}+V^{q+1}+1=0$, by setting $U^{(q+1)/m}=X$ and $V^{(q+1)/m}=Y$. 
In this extension only the places in $\mathcal{O}:=\{P_{(\alpha_i,0)}, P_{(0,\alpha_i)}, P_\infty^i \mid i = 1,\ldots, m\}$ have an interesting ramification behaviour, meaning that all the other places in $\mathcal{F}_m$ are unramified.

This gives rise to the following interesting property that we will use intensively later on in the paper. Namely we will use the fact that when looking at $X-a$ and $Y-b$ for $a^m,b^m\neq -1$, and denoting with $Q=Q_{(A,B)}$ a place of $\mathcal{F}_{q+1}$ above $P=P_{(a,b)}$, then since $P$ is unramified, we can view $X-a$ and $Y-b$ as functions in $\mathcal{F}_{q+1}$ instead of $\mathcal{F}_m$. 

The following two diagrams summarize the distribution of places in the function fields mentioned so far. Here we use the letter $Q$ to denote the places of the function field $\mathcal{F}_{q+1}$, and the letter $\bar{P}$ for the places in the intermediate function field $\mathbb{F}_{q^2}(X,V)$, in a similar way as we used the letter $P$ for the places in $\mathcal{F}_m$. 

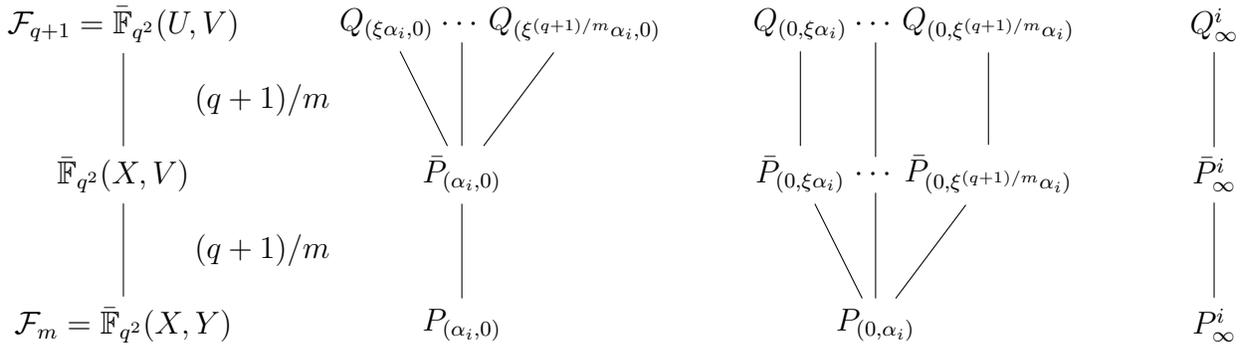
\begin{center}
    \begin{tikzpicture}
        \node (Q) at (-1,-2) {$\mathcal{F}_{q+1} = \bar{\F}_{q^2}(U,V)$};
        \node (I) at (-1,-4) {$\bar{\F}_{q^2}(X, V)$};
        \node (M) at (-1,-6) {$\mathcal{F}_m = \bar{\F}_{q^2}(X,Y)$};
        \draw (Q) -- (I) node[midway,right=8mm] {$(q+1)/m$};
        \draw (I) -- (M) node[midway,right=8mm] {$(q+1)/m$};

        \node (Q1) at (2.5,-2) {$Q_{(\xi \alpha_i,0)}$};
        \node (Q2) at (3.5,-2) {$\cdots$};
        \node (Q3) at (5,-2) {$Q_{(\xi^{(q+1)/m} \alpha_i,0)}$};

        \node (Q4) at (8,-2) {$Q_{(0,\xi \alpha_i)}$};
        \node (Q5) at (9,-2) {$\cdots$};
        \node (Q6) at (10.5,-2) {$Q_{(0,\xi^{(q+1)/m} \alpha_i)}$};

        \node (Q7) at (13.5,-2) {$Q_\infty^i$};

        \node (I1) at (3.5,-4) {$\bar{P}_{(\alpha_i,0)}$};

        \node (I2) at (8,-4) {$\bar{P}_{(0,\xi \alpha_i)}$};
        \node (I3) at (9,-4) {$\cdots$};
        \node (I4) at (10.5,-4) {$\bar{P}_{(0,\xi^{(q+1)/m} \alpha_i)}$};

        \node (I5) at (13.5,-4) {$\bar{P}_\infty^i$};

        \node (M1) at (3.5,-6) {$P_{(\alpha_i,0)}$};

        \node (M2) at (9,-6) {$P_{(0,\alpha_i)}$};

        \node (M3) at (13.5,-6) {$P_\infty^i$};

        \draw (Q1) -- (I1);
        \draw (Q2) -- (I1);
        \draw (Q3) -- (I1);

        \draw (Q4) -- (I2);
        \draw (Q5) -- (I3);
        \draw (Q6) -- (I4);

        \draw (I1) -- (M1);

        \draw (I2) -- (M2);
        \draw (I3) -- (M2);
        \draw (I4) -- (M2);

        \draw (Q7) -- (I5);
        \draw (I5) -- (M3);
    \end{tikzpicture}
\end{center}

We note also that in the above diagrams $\xi$ denotes a fixed primitive $(q+1)/m$-th root of unity, $\zeta_m$ is a primitive $m$-th root of unity and $a$ represents any element of $\bar{\mathbb{F}}_{q^2}$ such that $a^m+1 \ne 0$. 

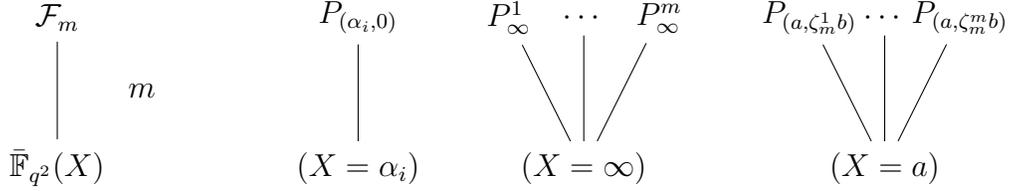
\begin{center}
    \begin{tikzpicture}
        \node (B) at (-1,-2) {$\mathcal{F}_m$};
        \node (C) at (-1,-4) {$\bar{\F}_{q^2}(X)$};
        \draw (B) -- (C) node[midway,right=8mm] {$m$};
        
        \node (H) at (3,-2) {$P_{(\alpha_i,0)}$};
        \node (I) at (3,-4) {$(X=\alpha_i)$};
        \draw (H) -- (I);
        
        \node (E1) at (5,-2) {$P_\infty^1$};
        \node (E2) at (6,-2) {$\cdots$};
        \node (E3) at (7,-2) {$P_\infty^m$};
        \node (F) at (6,-4) {$(X=\infty)$};

        \draw (E1) -- (F);
        \draw (E2) -- (F);
        \draw (E3) -- (F);

        \node (K1) at (9,-2) {$P_{(a,\zeta_m^1 b)}$};
        \node (K2) at (10,-2) {$\cdots$};
        \node (K3) at (11,-2) {$P_{(a,\zeta_m^m b)}$};
        \node (L) at (10,-4) {$(X=a)$};
        
        \draw (K1) -- (L);
        \draw (K2) -- (L);
        \draw (K3) -- (L);
    \end{tikzpicture}
\end{center}

The automorphism group of $\mathcal{F}_m$ has been computed in \cite{Le}. We recall its structure in the following lemma.

\begin{lemma}
\label{lemma:aut_F_m}
    Let $m>3$ be a proper divisor of $q+1$, and let $\zeta_m$ be a primitive $m$-th root of unity. Define
    \begin{align*}
        A_m := \{A_{a,b} \ : \ (X,Y) \mapsto (\zeta_m^a X, \zeta_m^b Y) \mid a,b = 0,\ldots, m-1\} \text{  and  } H_m := \langle S, T \rangle,
    \end{align*}
    where 
    \begin{align*}
        S \ : \ (X,Y) \mapsto \left( \frac{Y}{X}, \frac{1}{X}\right), \text{ and } T \ : \ (X,Y) \mapsto \left( \frac{1}{X}, \frac{Y}{X}\right).
    \end{align*}
    Then we have that $\text{Aut}(\mathcal{F}_m) =G_m$ where $G_m$ denotes the group $G_m:=\langle A_m, H_m\rangle$. In particular $\text{Aut}(\mathcal{F}_m)$ is isomorphic to a semidirect product of an abelian group of order $m^2$ (direct product of two cyclic groups of order $m$) and a symmetric group on $3$ letters.
    Furthermore, we have the relations 
    \begin{align*}
        S^3 = T^2 = 1,\ T^{-1}ST = S^{-1},\ S^{-1}A_{a,b}S = A_{b-a,-a} \text{ and } T^{-1}A_{a,b}T = A_{-a,b-a}.
    \end{align*}
\end{lemma}

Since $\mathcal{F}_m$ is $\mathbb{F}_{q^2}$-maximal the so-called Fundamental Equation \cite[Page xvii (ii)]{HKT} applies to $\mathcal{F}_m$. We summarize this property in the following theorem.

\begin{theorem}
\label{fundamentalequation}    
Let $P,Q$ be places of $\mathcal{F}_m$ and assume that $Q$ is $\mathbb{F}_{q^2}$-rational. Denote with $\Phi(P)$ the $\mathbb{F}_{q^2}$-Frobenius image of the place $P$. Then there exists a function $f_{P,Q} \in \mathcal{F}_m$ such that
$$(f_{P,Q})_{\mathcal{F}_m}=qP+\Phi(P)-(q+1)Q.$$
In particular if $P$ is also $\mathbb{F}_{q^2}$-rational then
$$(f_{P,Q})_{\mathcal{F}_m}=(q+1)P-(q+1)Q.$$
\end{theorem}

\subsection{Weierstrass semigroups and numerical semigroups.}

%We recall that for a numerical semigroup $G \subseteq \mathbb{N}$, the genus of $G$ is $g(G) := |\mathbb{N} \backslash G|$. 
%
%An important class of numerical semigroups are the telescopic semigroups, of which we recall the definition in the following \cite{HLP}. Let $(a_1, \ldots , a_k)$ be a sequence of positive integers with greatest common divisor equal to $1$. Define
%
%$$d_i := gcd(a_1,\ldots, a_i)$$
%and 
%$$A_i :=\bigg\{ \frac{a_1}{d_i},\ldots, %\frac{a_i}{d_i} \bigg\},$$
%for $i = 1,\ldots , k$. Let $d_0 := 0$ and $G_i$ be the semigroup generated by $A_i$. If $a_i/d_i \in G_{i-1}$ for all $i = 2,\ldots , k$ then the sequence $(a_1, \ldots, a_k)$ is called telescopic. A numerical semigroup is called telescopic if generated by a telescopic sequence.
%
%From \cite{HLP}*{Proposition 5.35} the genus of a telescopic semigroup $S$ generated by a telescopic sequence $(a_1, \ldots, a_k)$ is
%
%\begin{equation}
% \label{gen:tel}   
%g(S) =\frac{1}{2}\bigg( 1+ \sum_{i=1}^k\bigg( %\frac{d_{i-1}}{d_i}-1\bigg)a_i \bigg).
%\end{equation}

Given a rational place $P$ of function field $\mathcal{F}$, the Weierstrass semigroup $H(P)$ is the set of (non-negative) integers $k$ such that there exists a function in $\mathcal{F}$ with pole divisor $kP$. The set $G(P) = \mathbb{N} \backslash H(P)$ is called the set of gaps for $P$, and The Weierstrass Gap Theorem states that there are exactly $g$ such gaps, where $g$ is the genus of the function field $\mathcal{F}$. Now assume that the constant field of $\mathcal{F}$ is algebraically closed. Then it is known that for all but a finite number of places of the function field (called the Weierstrass places), the Weierstrass semigroup is the same. 

Clearly determining the Weierstrass semigroup $H(P)$ at a place $P$ of $\mathcal{F}$ is equivalent to determining the set of gaps $G(P)$. To compute the elements of $G(P)$ the following lemma will be used, see \cite{VS}*{Corollary 14.2.5}.

\begin{lemma}
\label{lemma:holomorphic_diffs}
An integer $n \in \mathbb{N}$ is a gap at a place $P$ if and only if there exists a holomorphic differential $\omega$ in $\mathcal{F}$ such that $v_P(\omega) = n - 1$.
\end{lemma}

It is not difficult to construct holomorphic differentials on the function field $\mathcal{F}_m$. An example, that will be useful later is $w:=dY/X^{m-1}$. This differential is holomorphic and has divisor 

\begin{equation}
 \label{div:w}   
(w)=(m-3) \sum_{i=1}^m P_{\infty}^i.
\end{equation}
To prove this, it is enough to consider the Kummer extension $\mathcal{F}_m / \overline{\mathbb{F}}_{q^2}(Y)$ of degree $m$
and using \cite{Sti}*{Corollary 3.4.7} to obtain 

$$
    (dY) = -2(Y)_\infty + \text{Diff}(\mathcal{F}_m/ \overline{\mathbb{F}}_{q^2}(Y))= -2 \sum_{i=1}^m P_{\infty}^i+ \sum_{i=1}^{m}(m-1)P_{(0,\alpha_i)}.
$$
 Finally combining the above with Equation \eqref{div:x}, gives Equation \eqref{div:w}. 

\section{Weierstrass semigroups in $\mathcal{F}_m$, $m \mid (q+1)$}\label{sec:three}

This section is divided into two parts, namely one where we compute the Weierstrass semigroups for every place $P\in\mathcal{O}$ and any $m \mid (q+1)$, and then a second part where we analyze the Weierstrass semigroup at places $P\not\in\mathcal{O}$ in the special cases where $m=\frac{q+1}{2}$ and $m = \frac{q+1}{3}$.

Before diving into it, we note that $g(\mathcal{F}_2)=0$ and $g(\mathcal{F}_3)=1$. Since the Weierstrass semigroups in these cases are trivially known, we will always assume without loss of generality that $m\geq 4$. 

\subsection{Weierstrass semigroup $H(P)$ for $P\in \mathcal{O}$}

In this section, we will find the Weierstrass semigroup for $P\in\mathcal{O}:=\{P_{(\alpha_i,0)}, P_{(0,\alpha_i)}, P_\infty^i \mid i = 1,\ldots, m\}$. 

We start by observing that $\mathcal{O}$ is in fact the $\text{Aut}(\mathcal{F}_m)$-orbit of $P_{(\alpha_1,0)}$. This observation is crucial because it implies that $H(P)=H(P_{(\alpha_1,0)})$ for all $P\in\mathcal{O}$, and hence that it is enough to compute $H(P_{(\alpha_1,0)})$.% to know the behaviour of Weierstrass semigroups in $\mathcal{O}$.

\begin{remark}
The fact that $\mathcal{O}$ is the $\text{Aut}(\mathcal{F}_m)$-orbit of $P_{(\alpha_1,0)}$ can be proven by using Lemma \ref{lemma:aut_F_m}. In fact using Lemma~\ref{lemma:aut_F_m} one can easily see that the $A_m$-orbit of $P_{(\alpha_1,0)}$ is
$\{P_{(\zeta_m^a \alpha_1,0)} \mid a = 0,\ldots, m-1\}=\{P_{(\alpha_j,0)} \mid j = 1,\ldots, m\}$, while the $A_m$-orbit of $P_{(0,\alpha_1)}$ is analogously $\{P_{(0,\alpha_j)} \mid a = 0,\ldots, m-1\}$. Since $S(P_{(\alpha_1,0)})=P_{(0,\alpha_1)}$ this already shows that $H(P_{(\alpha_j,0)})=H(P_{(0,\alpha_k)})$ for all $j,k=1,\ldots,m$. 
 Moreover, $S$ also maps the set of poles $\{P_{\infty}^i \mid i = 1,\ldots, m\}$ to $\{P_{(0,\alpha_i)} \mid i = 1,\ldots, m\}$, implying that in fact 
    \begin{align*}
        H(P_{(0,\alpha_i)}) = H(P_{(\alpha_j,0)}) = H(P_\infty^k) \text{ for all } i,j,k = 1,\ldots, m.
    \end{align*}
\end{remark}

We are then ready for the first theorem. Note that this theorem is not new, since it for example also follows directly from the results given in \cite{boseck}. However, for the sake of completeness we state it and give a short proof.
\begin{theorem}
    Let $m$ be a divisor of $q+1$. Then for a place $P\in\mathcal{O}$ we have the Weierstrass semigroup
    \begin{align*}
        H(P) = \langle m-1,m \rangle.
    \end{align*}
\end{theorem}

\begin{proof}
 Using Equations \eqref{div:y} and \eqref{div:xalphai} we see that the function $\frac{Y}{X-\alpha_1}$ has principal divisor $\sum_{i=2}^{m} P_{(\alpha_i,0)} - (m-1)P_{(\alpha_1,0)}$ and $\frac{1}{X-\alpha_1}$ has principal divisor $\sum_{i = 1}^{m} P_{\infty}^i - mP_{(\alpha_1,0)}$. This already shows that $\langle m-1,m \rangle \subseteq H(P_{(\alpha_1,0)})$. It is easily seen that the semigroup $\langle m-1,m \rangle$ has exactly $g(\mathcal{F}_m)$ many gaps using for example Proposition 5.11 from \cite{handbook}. Hence $\langle m-1,m \rangle = H(P_{(\alpha_1,0)})$.
\end{proof}

\subsection{Weierstrass semigroup $H(P)$ for $P\in\mathbb{P}_{\mathcal{F}_m}\setminus \mathcal{O}$}\label{sec:P_not_in_O}
In this section we consider the case $P\in\mathbb{P}_{\mathcal{F}_m}\setminus \mathcal{O}$ for $m=(q+1)/2$ and $m=(q+1)/3$. Instead of directly finding the Weierstrass semigroup $H(P)$, we will determine the gap sequences $G(P)$ using Lemma \ref{lemma:holomorphic_diffs}.

Using the same notation as in the previous section, we can write the place $P$ as $P:=P_{(a,b)}\in\mathbb{P}_{\mathcal{F}_m}\setminus \mathcal{O}$ with $a^m+b^m+1=0$ and $a\cdot b\neq 0$. Note that $b^m \neq -1$ and $a^m \neq -1$ in this setting. 
To ease notation, we also define the divisor $D_\infty := \sum_{i=1}^m P_\infty^i$. 

Let $Q:=Q_{(A,B)}$ be a place of $\mathcal{F}_{q+1}$ above $P$, that is $A^{q+1}+B^{q+1}+1=0$, $A^{(q+1)/m}=a$, $B^{(q+1)/m}=b$, and recall that $P$ does not ramify in $\mathcal{F}_{q+1}/\mathcal{F}_m$. The functions $\tilde{T} = U-A$ and $T=\frac{U-A}{A}$ are two local parameters at $Q=Q_{(A,B)}$. As in the proof of Proposition 2.12 in \cite{vicino}, a direct computation, using the equation $U^{q+1}+V^{q+1}+1=0$, gives that
    \begin{align*}%\label{eq:V-B}
        V-B = -\frac{A^q}{B^q} \tilde{T} + O(\tilde{T}^q),
    \end{align*}
    and 
    \begin{align*}%\label{eq:(V-B)/B}
        \frac{V-B}{B} = -\frac{A^{q+1}}{B^{q+1}} T + O(T^q).
    \end{align*}
Here the big-$O$ notation $O(\tilde{T}^\ell)$ (resp. $O(T^\ell)$) is used to indicate terms in the power series development of a function in $\tilde{T}$ (resp. $T$) of valuation at least $\ell$. 

We define the function $F_P:= f_{P,P_{(\alpha_1,0)}}\cdot (X-\alpha_1)^{(q+1)/m}$, where $f_{P,P_{(\alpha_1,0)}}$ is the function from Theorem \ref{fundamentalequation}. Then using Theorem \ref{fundamentalequation} and Equation \eqref{div:xalphai} we see that $F_P$ has divisor 
    \begin{equation}
    \label{div:Fpnr}
        (F_P) = qP + \Phi(P) - \frac{q+1}{m} D_\infty.
    \end{equation}
In particular if $P$ is $\F_{q^2}$-rational, we have 
\begin{equation}
    \label{div:Fp}
        (F_P) = (q+1) P - \frac{q+1}{m} D_\infty.
\end{equation}

As mentioned, the goal is to create a family of holomorphic differentials with specific valuation, and for that we need some special functions. Therefore define 
\begin{equation}\label{def:f0tilde}
    \tilde{f}_0 := a^{m-1}(X-a) + b^{m-1}(X-b).
\end{equation}
Our first goal is to compute the power series expansions of $X-a$, $Y-b$ and $\tilde{f}_0$ with respect to the local parameter $\tilde{T}$ defined in the previous section. 

When seeing $X-a$ and $Y-b$ as functions in $\mathcal{F}_{q+1}$ instead of $\mathcal{F}_m$, we can compute the power series with respect to the local parameter $\tilde{T}=U-A$ at $Q$ as follows:
\begin{align*}
    X-a &= U^{(q+1)/m} - A^{(q+1)/m} = (U - A + A)^{(q+1)/m} - A^{(q+1)/m} \\
    & = \sum_{j=0}^{(q+1)/m}{(q+1)/m \choose j} (U-A)^j A^{(q+1)/m-j} - A^{(q+1)/m} \\
    & = \sum_{j=1}^{(q+1)/m}{(q+1)/m \choose j} (U-A)^j A^{(q+1)/m-j}.
\end{align*}
Using that $a=A^{(q+1)/m}$ then gives us the expression
\begin{align}\label{eq:pow_ser_exp_x_a}
    \frac{X-a}{a} & %= \sum_{j=1}^{(q+1)/m}{(q+1)/m \choose j} (U-A)^j A^{(q+1)/m-j-(q+1)/m} 
    = \sum_{j=1}^{(q+1)/m}{(q+1)/m \choose j} \tilde{T}^j A^{-j}. 
\end{align}
In a completely similar manner, we obtain
\begin{align*}
    Y-b = \sum_{j=1}^{(q+1)/m}{(q+1)/m \choose j} (V-B)^j B^{(q+1)/m-j},
\end{align*}
and
\begin{align*}
    \frac{Y-b}{b} = \sum_{j=1}^{(q+1)/m}{(q+1)/m \choose j} (V-B)^j B^{-j}.
\end{align*}
Using that $V-B = -\frac{A^q}{B^q} \tilde{T} + O(\tilde{T}^q)$, we get
\begin{align}\label{eq:pow_ser_exp_y_b}
    \frac{Y-b}{b} = %\sum_{j=1}^{(q+1)/m}{(q+1)/m \choose j} \left(-\frac{A^q}{B^q} \tilde{T} + O(\tilde{T}^q)\right)^j B^{-j} = 
    \sum_{j=1}^{(q+1)/m}{(q+1)/m \choose j} \left(-\frac{A^q}{B^q} \tilde{T}\right)^j B^{-j}+O(\tilde{T}^q).
\end{align}
Now, this allows us to obtain an expression for $\tilde{f}_0$, namely 
\begin{align*}
    \tilde{f_0} 
    & = a^{m-1}(X-a) + b^{m-1}(Y-b) 
     = a^m \frac{X-a}{a} + b^m \frac{Y-b}{b}\\
%    & = A^{q+1} \left(\sum_{j=1}^{(q+1)/m}{(q+1)/m \choose j} \tilde{T}^j A^{-j}  \right) + B^{q+1} \left(\sum_{j=1}^{(q+1)/m}{(q+1)/m \choose j} \left(-\frac{A^q}{B^q} \tilde{T} + O(\tilde{T}^q)\right)^j B^{-j} \right) \\
    & = \sum_{j=1}^{(q+1)/m}{(q+1)/m \choose j} \left(A^{q+1-j} + B^{q+1-j} \left(-\frac{A^q}{B^q} \right)^j \right) \tilde{T}^j + O(\tilde{T}^q)\\
    & = \sum_{j=1}^{(q+1)/m}{(q+1)/m \choose j} \left(A^{q+1-j} + (-1)^j\frac{A^{qj}}{B^{(j-1)(q+1)}}  \right) \tilde{T}^j + O(\tilde{T}^q).
\end{align*}
Note that the coefficient for $\tilde{T}$ in the above expansion is 
\begin{align*}
    {(q+1)/m \choose 1} \left(A^{q+1-1} -A^{q}  \right) = 0,
\end{align*}
implying that $v_P(\tilde{f}_0) \geq 2$.  In fact, using $A^{q+1}+B^{q+1}+1=0$, we see that the coefficient for $\tilde{T}^2$ is
\begin{align*}
    {(q+1)/m \choose 2} \left(A^{q-1} + \frac{A^{2q}}{B^{q+1}} \right) 
    &= {(q+1)/m \choose 2} \left(\frac{A^{q-1}B^{q+1} + A^{2q}}{B^{q+1}} \right) \\
 %   &= {(q+1)/m \choose 2}\left(\frac{A^{q-1}(-1-A^{q+1}) + A^{2q}}{B^{(q+1)}} \right) \\
    &= -{(q+1)/m \choose 2}\frac{A^{q-1}}{B^{q+1}},
\end{align*}
This coefficient is non-zero, since the assumption $a \cdot b \neq 0$ implies that $A\cdot B \neq 0$. This proves that $v_P(\tilde{f}_0) = 2$, and 
\begin{equation}
\label{power:f0}
    \tilde{f_0} 
    =   -{(q+1)/m \choose 2}\frac{A^{q-1}}{B^{q+1}}\tilde{T}^2 + \sum_{j=3}^{(q+1)/m}{(q+1)/m \choose j} \left(A^{q+1-j} + (-1)^j\frac{A^{qj}}{B^{(j-1)(q+1)}}  \right) \tilde{T}^j + O(\tilde{T}^q). 
\end{equation}

At this point, we will treat the case $m=(q+1)/2$ and $m=(q+1)/3$ in two separate subsections.

\subsubsection{The case $m=(q+1)/2$}

In this section we will restrict our considerations to the case $m=\frac{q+1}{2}$. In this specific case, to find $H(P)$, it will be enough to use the functions constructed so far in Section 3 in combination with the differential $w$ from Equation \eqref{div:w} (to construct holomorphic differentials). 
\begin{theorem}\label{thm:gaps_m2}
    Let $P\in\mathbb{P}_{\mathcal{F}_m}\backslash \mathcal{O}$ be an $\F_{q^2}$-rational place, with $m=\frac{q+1}{2}$.
    Then
    \begin{align*}
        G(P) = \{i + 2j + k(q+1) + 1 \mid i,j,k \in\mathbb{Z}_{\geq 0}, \  m-3-i-j-k(q+1)/m\geq 0\}.
    \end{align*}
\end{theorem}
\begin{proof}
    If we can create a holomorphic differential with valuation $i + 2j + k(q+1)$, then Lemma~\ref{lemma:holomorphic_diffs} gives us that $i + 2j + k(q+1) + 1$ is a gap.
    Therefore, let $w$ be the differential considered in Equation \eqref{div:w}, and let $i,j,k$ be non-negative integers. We define the function
    \begin{align*}
        h_{i,j,k} := (X-a)^i \tilde{f_0}^j F_P^k.
    \end{align*}
   Combining Equations \eqref{div:w}, \eqref{power:f0} and \eqref{div:Fp} shows that the differential $h_{i,j,k}w$ has divisor:
    \begin{align*}
        (h_{i,j,k}w) &= i (X-a) + j (\tilde{f}_0) + k (F_P) + (w)\\
        &= (i + 2j + k(q+1))P + E_{i,j,k} + \left(m-3 - i - j - k\frac{q+1}{m}\right)D_\infty,
    \end{align*}
    where $E_{i,j,k}$ is a positive divisor collecting the zeros other than $P$ of the functions $X-a$ and $\tilde{f}_0$. 
    We see that if $m-3 - i - j - k\frac{q+1}{m} \geq 0$, then $(h_{i,j,k}w)$ is holomorphic and $i + 2j + k(q+1) + 1$ is a gap at $P$.
    This shows that the set 
    $$\mathcal{T}:=\left\{i + 2j + k(q+1) + 1 \mid m-3 - i - j - k\frac{q+1}{m} \geq 0 \right\}$$
    is contained in $G(P)$ and we can deduce that equality holds by proving that $|\mathcal{T}|=g(\mathcal{F}_m)$. Note that the assumption $m=(q+1)/2$ implies that $g(\mathcal{F}_m)= (q-1)(q-3)/8$. 
    
    Now note that 
    \begin{align*}
        \mathcal{T} = \left\{i + 2j + k(q+1) + 1 \mid i=0,1;\ j = 0,\ldots, \frac{q-5}{2}-i;\ m-3 - i - j - k\frac{q+1}{m} \geq 0 \right\},
    \end{align*} 
    and that with the posed restrictions on $i$, $j$ and $k$ no repetitions occur in the corresponding gaps for $P$. Hence    
    $|\mathcal{T}|=|S|$, where
    \begin{align*}
        S := \left\{(i,j,k) \mid i=0,1,\ j = 0,\ldots \frac{q-5}{2}-i,\ k = 0,\ldots, \left\lfloor \frac{q-5-2i-2j}{4} \right\rfloor \right\}.
    \end{align*}
    Hence it is enough to show that $|S|=g(\mathcal{F}_m)$. One has

    \begin{align} \label{formulaS} 
        |S| &= \sum_{i=0}^1 \sum_{j=0}^{\frac{q-5}{2}-i} \left( \left\lfloor \frac{q-5-2j-2i}{4} \right\rfloor + 1\right)\notag\\
        &=\sum_{i=0}^1 \sum_{j=0}^{\frac{q-5}{2}-i} \left\lfloor \frac{q-5-2j-2i}{4} \right\rfloor +q-4.
    \end{align}
 
For given $i=0,1$ let $N:=\frac{q-5}{2}-i$. We start by computing the inner sum
    \begin{align*}
\sum_{j=0}^{\frac{q-5}{2}-i} \left\lfloor \frac{q-5-2j-2i}{4} \right\rfloor &=\sum_{j=0}^{N} \left\lfloor \frac{2N-2j}{4} \right\rfloor=\sum_{j=0}^{N} \left\lfloor \frac{N-j}{2} \right\rfloor\\ &=\sum_{k=0}^{N} \left\lfloor \frac{k}{2} \right\rfloor=\sum_{s=0}^{\lfloor \frac{N}{2}\rfloor}s+ \sum_{s=0}^{\lfloor\frac{N-1}{2}\rfloor} s\\ &=\begin{cases} \frac{N^2}{4} , & \text{if} \ N \equiv 0 \pmod 2, \\ \frac{N-1}{2}\bigg(\frac{N-1}{2}+1 \bigg), & \text{if} \ N \equiv 1 \pmod 2.\end{cases}
   \end{align*}

Using the formula above and noting that if $q \equiv 3 \pmod 4$ (resp. $q \equiv 1 \pmod 4$) then $N=\frac{q-5}{2}-i$ is even for $i=0$ and odd for $i=1$ (resp. even for $i=1$ and odd for $i=0$), Equation \eqref{formulaS} gives

\begin{align*}
|S|&=
\begin{cases}  \frac{(q-5)^2}{4} + \frac{\frac{q-5}{2}-2}{2}\bigg(  \frac{\frac{q-5}{2}-2}{2}+1\bigg)+q-4, & \text{if} \ q \equiv 3 \pmod 4, \vspace{3mm}\\ \frac{\frac{q-5}{2}-1}{2}\bigg(  \frac{\frac{q-5}{2}-1}{2}+1\bigg)+\frac{\left( \frac{q-5}{2}-1\right)^2}{4}+q-4, & \text{if} \ q \equiv 1 \pmod 4.\end{cases} \\
&=\frac{(q-1)(q-3)}{8}=g(\mathcal{F}_m).
  \end{align*}

\end{proof}
Now, for a non-rational place $P$, the situation looks quite similar, as the following corollary shows. 
\begin{corollary}
    Let $P\in \mathbb{P}_{\mathcal{F}_m}\backslash\mathcal{O}$ be a non $\F_{q^2}$-rational place, and $m=\frac{q+1}{2}$. 
    Then
$$
        G(P) = \{i + 2j + kq + 1 \mid i,j,k \in \mathbb{Z}_{\geq 0}, m-3-i-j-k(q+1)/m\geq 0\}. 
$$

\end{corollary}
\begin{proof}
    The proof is exactly analogous to that of Theorem~\ref{thm:gaps_m2}. This time the function $h_{i,j,k} := (X-a)^i \tilde{f_0}^j F_P^k$ instead has valuation $i + 2j + kq$, since $v_P(F_P)=q$ for $P$ non-rational by Equation \eqref{div:Fpnr}. 
    Hence the set of constructed gaps using this functions this time is $\mathcal{T}':= \{i + 2j + kq + 1 \mid i,j,k \in\mathbb{Z}_{\geq 0}, m-3-i-j-k(q+1)/m\geq 0\}$. If we can argue that $|S|=|\mathcal{T}'|$ for the set $S$ defined in the Theorem~\ref{thm:gaps_m2}, then we are done. 
    
    Since by definition $S = \left\{(i,j,k) \mid i=0,1,\ j = 0,\ldots \frac{q-5}{2}-i,\ k = 0,\ldots, \left\lfloor \frac{q-5-2i-2j}{4} \right\rfloor \right\}$, to prove that  $|S|=|\mathcal{T}'|$ it is sufficient to show that restricting our parameters to the ranges $i=0,1,\ j = 0,\ldots \frac{q-5}{2}-i,\ k = 0,\ldots, \left\lfloor \frac{q-5-2i-2j}{4} \right\rfloor$, gives acceptable parameters according to the restrictions in $\mathcal{T}'$ and that no repetitions in $\mathcal{T}'$ occur for different choices of such parameters. 
    This can directly be shown by hand assuming by contradiction that different choices of the parameters $(i,j,k)$ give equal values $i + 2j + kq + 1$.
\end{proof}

Combining all the results in this subsection we now have $H(P)$ for all places $P$ of $\mathcal{F}_m$, with $m=(q+1)/2$.
In the next section we will try to do something similar for $m=\frac{q+1}{3}$, obtaining however a very different and more complicated behaviour of the semigroups $H(P)$.

\subsubsection{The case $m=(q+1)/3$} 

In this section we will determine $G(P)$ in the case $m=\frac{q+1}{3}$, with  $P:=P_{(a,b)}\in\mathbb{P}_{\mathcal{F}_m}\setminus \mathcal{O}$ with $a^m+b^m+1=0$, $a \cdot b\neq 0$. As before, this implies that $b^m \neq -1$ and $a^m \neq -1$. The overall strategy will again be to construct holomorphic differentials with a certain valuation and use Lemma \ref{lemma:holomorphic_diffs}. 
This time however, we will need a lot more functions than before. For this, we will use techniques similar to those in \cite{vicino}*{Section 3}, and several results from the same paper. In the following we will list the most relevant results from \cite{vicino} for our purposes, omitting the corresponding proofs.

Let
\begin{align*}
    \alpha(P_{(a,b)}) := \frac{a^m}{1+a^m}.
\end{align*}
For the rest of this section we will simply write $\alpha = \alpha(P_{(a,b)})$ for simplicity. 
The next definition is a special case of~\cite[Definition 3.1]{vicino}. 
Since we are working over $\mathbb{F}_{q^2}$ and $m=\frac{q+1}{3}$, we know that there exists a primitive cube root of unity in $\mathbb{F}_{q^2}$, which we will denote by $\zeta_3$.
Then for every $i\in\mathbb{Z}$ we define the functions
\begin{align*}
    \mathcal{P}_i(s):=\frac{(s+\zeta_3)^{3i}-(s+\zeta_3^2)^{3i}}{3(\zeta_3-\zeta_3^2)s(s-1)},
\end{align*}
and
\begin{align*}
    \mathcal{Q}_i(s) := \frac{\left(\frac{1-\zeta_3}{3}\right)(s+\zeta_3)^{3i-1} + \left(\frac{1-\zeta_3^2}{3}\right)(s+\zeta_3^2)^{3i-1} }{s-1}.
\end{align*}
In \cite[Example 3.2]{vicino}, we have an explicit description of these functions for small values of $i$:

$$
    \pc_0(s) = 0, \ \pc_1(s) = 1, \ \pc_2(s) = 2s^3-3s^2-3s+2, %\ \pc_3(s) = 3s^6 - 9s^5-9s^4+33s^3-9s^2-9s+3,
$$
and
$$
    \qc_0(s) = (s^2-s+1)^{-1}, \ \qc_1(s) = s+1, \ \qc_2(s) = s^4+s^3-9s^2+s+1.$$
    %$$\qc_3(s) = s^7+s^6-27s^5+29s^4+29s^3-27s^2+s+1.$$
The following lemmas from \cite{vicino} summarize some important properties of the expressions $\pc_i(s)$ and $\qc_i(s)$.
\begin{lemma}[{\cite[Lemma 3.3]{vicino}}]
    If $i$ is a positive integer, then $\pc_i(s)$ is a nonzero polynomial in $s$ of degree at most $3i-3$, and $\qc_i(s)$ is a nonzero polynomial in $s$ of degree $3i-2$.
\end{lemma}
%\begin{lemma}[{\cite[Lemma 3.5]{vicino}}]\label{lemma:3.5}
%    For any three integers $i,j,\ell\in \mathbb{Z}$, we have the following relations 
%    \begin{enumerate}
%        \item[(i)] $\pc_i(s)\pc_{\ell+j}(s) - \pc_j(s)\pc_{\ell+i}(s) = (s^2-s+1)^{3j}\pc_{i-j}(s)\pc_{\ell}(s)$,\label{eq:(3.6)}
%        \item[(ii)] $\pc_i(s)\qc_{\ell+j}(s)-\pc_j(s)\qc_{\ell+i}(s) = (s^2-s+1)^{3j}\pc_{i-j}(s)\qc_\ell(s).$\label{eq:(3.7)}
%    \end{enumerate}
%\end{lemma}
\begin{lemma}[{\cite[Remark 3.8]{vicino}}]\label{lemma:PQ_no_common_zeros}
    If $i\in\mathbb{Z}_{>0}$, then $\pc_i(s)$ and $\qc_i(s)$ have no common roots.
\end{lemma}
\begin{lemma}[{\cite[Remark 3.9 and Definition 3.10]{vicino}}]
    For any $s\in\bar{\F}_{q^2}\backslash\{0,1,-\zeta_3,-\zeta_3^2\}$, there exists an $i\in\mathbb{Z}_{>0}$ such that $\pc_{i+1}(s)=0$. % over $\bar{\F}_{q^2}$.
    The smallest such $i$ is called the $\pc$-order of $s$.
\end{lemma}

Now we are ready for the first of the two theorems in which we construct functions with a certain valuation. 
In both theorems, we will construct functions of valuation $3j+2$ for certain values of $j$, but in the first theorem we will assume that $\alpha^2-\alpha+1\neq 0$, whereas in the second we will assume that $\alpha^2-\alpha+1=0$.
The first theorem is inspired by~\cite[Theorem 3.12]{vicino}. Also its proof is quite similar to that of ~\cite[Theorem 3.12]{vicino}.

\begin{theorem}\label{thm:3.12}
    Let $P = P_{(a,b)}\in \mathbb{P}_{\mathcal{F}_m}\backslash\mathcal{O}$, $m=\frac{q+1}{3}$, and assume that $\alpha^2-\alpha+1\neq 0$. 
    Let $i$ denote the $\pc$-order of $\alpha$. 
    Then if $i> m-2$, there exist functions $f_j \in \mathcal{L}((j+1)D_\infty)$ such that $v_P(f_j) = 3j+2$ for every $j=0,\ldots m-2$. 
    If on the other hand, $i\leq m-2$, there exist functions $f_j \in \mathcal{L}((j+1)D_\infty)$ such that $v_P(f_j) = 3j+2$ for $j\in 0,\ldots, i-1$, as well as a function $f_i\in\mathcal{L}((i+1)D_\infty)$ such that $v_P(f_i) = 3i+3$. 
\end{theorem}
\begin{proof}
    Let $\pc_j:=\pc_j(\alpha)$ and $\qc_j = \qc_j(\alpha)$.
    Recall that $Q=Q_{(A,B)}$ is the place over $P$ in $\mathcal{F}_{q+1}$, and $T=\frac{U-A}{A}$ is a local parameter of $Q$. 
    We now claim that for every $0\leq j \leq \min\{m-2,i\}$, there is a function $f_j\in\mathcal{L}((j+1)D_\infty)$ such that the power series expansion of $f_j$ with respect to $T$ is 
    \begin{align}\label{eq:f_j_exp}
        f_j = 3 \pc_{j+1} T^{3j+2} + \qc_{j+1} T^{3j+3} + O(T^q).
    \end{align}
    Recall that by definition, $\pc_{i+1}=0$, and by Lemma~\ref{lemma:PQ_no_common_zeros} $\qc_{i+1}\neq 0$, so for $j=i$ we have 
    \begin{align*}
        f_i = \qc_{i+1} T^{3i+3} + O(T^q).
    \end{align*}
    Note that $3j+3 \leq 3(m-2)+3 = 3m -3 = q+1-3 = q-2 < q$, so Equation \eqref{eq:f_j_exp} would imply that $v_P(f_j)=v_Q(f_j) =3j+2$ for $0\leq j \leq \min\{m-2,i-1\}$, while we know from the above expression for $f_i$ that $v_P(f_i) =v_Q(f_i)=3i+3$. 
    Thus if we can find $f_j$ as in Equation \eqref{eq:f_j_exp}, then we are done.
    
    As in the proof of~\cite[Theorem 3.12]{vicino}, we will to create $f_0,\ f_1$ and $f_2$ explicitly, while afterwards one can create $f_j$ for $3\leq j\leq i$ recursively. As a matter of fact, the proof from \cite{vicino} can be taken verbatim as soon as we have constructed $f_0$. In other words: the construction of $f_1$, $f_2$ and $f_j$ for $j\ge 3$, as well as the remainder of the proof can be carried out in exactly the same way as the proof of~\cite[Theorem 3.12]{vicino}.
 
    Let $f_0 := \alpha^{-1} \tilde{f}_0$, with $\tilde{f}_0$ as in Equation \eqref{def:f0tilde}. %$\tilde{f}_0 = a^m \left(\frac{X-a}{a}\right) + b^m \left(\frac{Y-b}{b}\right)$ considered at the beginning of section 3. 
    Further denote by $x_a:=\frac{X-a}{a}$, and recall that $a=A^3$, $b=B^3$ as well as the local parameter $T=\frac{U-A}{A}$. Then substituting $m=(q+1)/3$ in the powers series obtained in \eqref{eq:pow_ser_exp_x_a} and \eqref{eq:pow_ser_exp_y_b} we obtain 
    \begin{align*}
        x_a = 3T+3T^2+T^3 \quad \text{and} \quad         \frac{Y-b}{b} = - 3\frac{A^{q+1}}{B^{q+1}} T + 3\frac{A^{2(q+1)}}{B^{2(q+1)}} T^2 - \frac{A^{3(q+1)}}{B^{3(q+1)}} T^3  + O(T^q).
    \end{align*}
    Using these power series expansions, Equation \eqref{def:f0tilde} and the fact that $f_0 = \alpha^{-1} \tilde{f}_0 =  \frac{a^m+1}{a^m}\tilde{f}_0 = \frac{A^{q+1}+1}{A^{q+1}}\tilde{f}_0$, the power series expansion of $f_0$ is readily found to be as follows:
%    and finally using that $a^m = A^{3m} = A^{q+1}$, $b^m = B^{q+1}$ and $B^{q+1} = -1-A^{q+1}$, and we get from Equation \eqref{power:f0}
%    \begin{align*}
%        \tilde{f}_0 & = \left(\frac{3A^{q+1}}{1+A^{q+1}}\right) T^2 + \left(A^{q+1} - \frac{A^{3(q+1)}}{(-1-A^{q+1})^2} \right) T^3  + O(T^q)\\
%        & = 3\alpha T^2 + \left(\frac{A^{q+1}+2A^{2(q+1)}}{(1+A^{q+1})^2} \right) T^3  + O(T^q). 
%    \end{align*}
%    Now we can compute the power series expansion of $f_0 = \alpha^{-1} \tilde{f}_0 =  \frac{a^m+1}{a^m}\tilde{f}_0 = \frac{A^{q+1}+1}{A^{q+1}}\tilde{f}_0$, namely
    \begin{align*}
        f_0 = 3 T^2 + \frac{1+2A^{q+1}}{1+A^{q+1}} T^3  + O(T^q)
        = 3 T^2 + (\alpha+1) T^3  + O(T^q)
        = 3 \pc_1 T^2 + \qc_1 T^3 + O(T^q).
    \end{align*}
    We see that Equation \eqref{eq:f_j_exp} indeed holds for $f_0$ and that $f_0 \in \mathcal{L}(D_\infty)$.  
 This last fact follows from the fact that $f_0$ and $\tilde{f}_0$ are linear combinations of $X-a$ and $Y-b$ (so only poles at $P_\infty^k$ can occur) and that for all $k=1,\ldots,m$, 
    \begin{align*}
        v_{P_\infty^k}(f_0) = v_{P_\infty^k}(\tilde{f}_0) 
        \geq \min\{ v_{P_\infty^k}(X-a),\ v_{P_\infty^k}(Y-b)\} 
        = \min\{ -1 , -1\} = -1.
    \end{align*}
    The remainder of the proof can be taken verbatim from that of \cite[Theorem 3.12]{vicino}.    
\end{proof}
Theorem~\ref{thm:3.12} will turn out to provide us with some very useful functions, but since it relies on the assumption that $\alpha^2-\alpha+1\neq 0$, the natural thing to do now is to bring an equivalent theorem that does not rely on this assumption. While Theorem \ref{thm:3.12} is similar to \cite[Theorem 3.12]{vicino}, the next theorem is inspired by \cite[Theorem 3.19]{vicino}. Also its proof is largely the same as \cite[Theorem 3.19]{vicino}. 
\begin{theorem}\label{thm:3.19}
    Assume that $\alpha^2 -\alpha + 1 = 0$. 
    Then for every integer $j = 0,\ldots m-2$, there exists a function $g_j\in\mathcal{L}((j+1)D_\infty)$ with $v_P(g_j)=3j+2$. 
\end{theorem}
\begin{proof}
    This proof will follow a similar structure to that of Theorem~\ref{thm:3.12}, however in this proof, we can use some of the results that have already been found. 
    We claim that for $j = 0,\ldots m-2$ there exist functions $g_j\in \mathcal{L}((j+1)D_\infty)$ with power series expansion 
    \begin{align*}%\label{eq:g_j_exp}
        g_j = 3T^{3j+2} + (\alpha+1)T^{3j+3} + O(T^q).
    \end{align*}
    Note that $3\pc_1=3$ and $\qc_1=(\alpha+1)$, so we already know of at least one function for which this holds, namely $f_0$, so we set $g_0:=f_0$. From this point on, the proof is exactly the same as the proof of \cite[Theorem 3.19]{vicino}.
\end{proof}
Theorems~\ref{thm:3.12} and~\ref{thm:3.19} imply that, no matter what $\alpha^2-\alpha+1$ is, the fact that if the $\pc$-order of $\alpha$ is larger than $m-2$, then there exist a function $f_j$ for every $0\leq j \leq m-2$ such that $f_j\in\mathcal{L}((i+1)D_\infty)$ and $v_P(f_j) = 3j+2$. In the case where the $\pc$-order of $\alpha$ is at most $m-2$, Theorem~\ref{thm:3.12} gives us that we have a function $f_i\in\mathcal{L}((i+1)D_\infty)$ with valuation $v_P(f_i) = 3i+3$. This brings us to the main theorem of this section.

\begin{theorem}\label{conj:gaps}
    Let $P:=P_{(a,b)}\in \mathbb{P}_{\mathcal{F}_m}\backslash \mathcal{O}$ and $m=\frac{q+1}{3}$. Moreover, let $i$ be the $\pc$-order of $\alpha=a^m/(1+a^m)$. Denote by $G(P)$ the gap sequence for $P$.
    
    If $i> m-2$, then 
    \begin{align*}
        G(P)=
        \left\{ k(q+1) + \ell_0 \cdot 2 + \ell_1 \cdot (3j+2) + s +1 \ \middle\vert\ 
        \begin{array}{l}
            j,k,\ell_0,\ell_1,s \geq 0,\ s \leq 1,\ j \leq m-2,\\
            3k + \ell_0 + (j+1)\ell_1 + s \leq m-3.
        \end{array}
        \right\}
    \end{align*}
    in case in $P$ is $\F_{q^2}$-rational, and 
    \begin{align*}
        G(P)=
        \left\{ kq + \ell_0 \cdot 2 + \ell_1 \cdot (3j+2) + s +1 \ \middle\vert\ 
        \begin{array}{l}
            j,k,\ell_0,\ell_1,s \geq 0,\ s \leq 1,\ j \leq m-2,\\
            3k + \ell_0 + (j+1)\ell_1 + s \leq m-3.
        \end{array}
        \right\}
    \end{align*}
    in case $P$ is not $\F_{q^2}$-rational.

    If instead $i\leq m-2$, then% we have the gap sequences 
    \begin{align*}
        G(P)
        =
        \left\{  k(q+1) + \ell_0\cdot 2 + \ell_1 \cdot (3j+2) + \ell_2\cdot(3i+3) + s + 1\
        \middle\vert\ k,\ell_0,\ell_1,\ell_2, s \geq 0,\ s \le 1, \right.\\
        \left.  \text{and }\ j \leq i-1,\ 3k + \ell_0 + \ell_1(j+1)+ \ell_2 (i+1)+ s\leq m-3
         \right\}
    \end{align*}
    for $P$ $\F_{q^2}$-rational, and 
    \begin{align*}
        G(P)
        =
        \left\{  kq + \ell_0\cdot 2 + \ell_1 \cdot (3j+2) + \ell_2\cdot(3i+3) + s + 1
        \middle\vert\ k,\ell_0,\ell_1,\ell_2, s \geq 0,\ s \le 1,\ j \leq i-1,\right. \\
        \left. \text{and }\ 3k + \ell_0 + \ell_1(j+1)+ \ell_2 (i+1)+ s\leq m-3
        \right\}
    \end{align*}
    for $P$ not $\F_{q^2}$-rational.
\end{theorem}

Since the proof of Theorem \ref{conj:gaps} is quite long, we present it in a separate section.

\section{Proof of Theorem \ref{conj:gaps}} \label{sec_four}

We divide the proof of Theorem \ref{conj:gaps} into two steps, each presented in a subsection. First of all we prove that all the claimed values above are in fact gaps in the different cases (Lemma \ref{theyaregaps}). Then we will show in Lemmas \ref{igtm2} and \ref{counting-ismall} that the found gaps are all the gaps, that is, the sets listed above all have cardinality $g(\mathcal{F}_{(q+1)/3})$. 

\subsection{Construction of gaps}

\begin{lemma} \label{theyaregaps}
All the values in Theorem \ref{conj:gaps} are gaps at $P_{(a,b)}\in \mathbb{P}_{\mathcal{F}_{(q+1)/3}}\backslash \mathcal{O}$
\end{lemma}

\begin{proof}
    We will once again show that the stated elements are gaps by constructing holomorphic differentials with certain valuations. Therefore, let $w$ be the differential whose divisor was determined in Equation \eqref{div:w}. 
As in the theorem, we distinguish the two cases $i>m-2$ and $i\leq m-2$, so for now assume that $i>2$. 
Theorems~\ref{thm:3.12} and~\ref{thm:3.19} ensure that for each $j=1,\ldots, m-2$ there exist functions $f_j\in\mathcal{L}((j+1)D_\infty)$ with $v_P(f_j) = 3j+2$. 
We then define the differentials
\begin{align*}
    h_{k,\ell_0,\ell_1,s}:= F_P^k f_0^{\ell_0} f_j^{\ell_1} x_a^s w,
\end{align*}
where $F_P$ is the function defined in Equation \eqref{div:Fp} and $x_a=\frac{X-a}{a}$. Since we know the valuations of each of the functions constituting $h_{k,\ell_0,\ell_1,s}$, we can compute 
\begin{align*}
    v_P(h_{k,\ell_0,\ell_1,s}) = 
    \begin{cases}
        k(q+1) + \ell_0 \cdot 2 + \ell_1 \cdot (3j+2) + s & \text{ if $P$ is $\F_{q^2}$-rational} \\
        kq + \ell_0 \cdot 2 + \ell_1 \cdot (3j+2) + s & \text{ if $P$ is not $\F_{q^2}$-rational} 
    \end{cases}.
\end{align*}
Lemma \ref{lemma:holomorphic_diffs} states that $v_P(h_{k,\ell_0,\ell_1,s})+1$ is a gap as long as $h_{k,\ell_0,\ell_1,s}$ is holomorphic, so we wish to bound the exponents such that $v_{P_\infty^k}(h_{k,\ell_0,\ell_1,s}) \geq 0$ for all $k=1,\ldots m$. 
Recall that $v_{P_\infty^k}(F_P)=\frac{q+1}{m}=3$ and that each $f_j\in\mathcal{L}((j+1)D_\infty)$ as well as $(w)=(m-3) D_\infty$. Then we get that for each $k=1,\ldots m$, 
\begin{align*}
    v_{P_\infty^k}(h_{k,\ell_0,\ell_1,s}) 
    &= -3k - \ell_0 - (j+1)\ell_1 - s + m-3.
\end{align*}
This means that $h_{k,\ell_0,\ell_1,s}$ is holomorphic exactly when $3k + \ell_0 + (j+1)\ell_1 + s \leq m-3$. 
This implies that for a rational place $P$, 
\begin{align*}
    \left\{ k(q+1) + \ell_0 \cdot 2 + \ell_1 \cdot (3j+2) + s +1 \ \middle\vert\ 
    \begin{array}{l}
        j,k,\ell_0,\ell_1,s \geq 0,\ s \leq 1,\ j \leq m-2,\\
        3k + \ell_0 + (j+1)\ell_1 + s \leq m-3.
    \end{array}
    \right\}
\end{align*}
are in fact gaps at $P$, while for a non-rational place $P$, 
\begin{align*}
    \left\{ kq + \ell_0 \cdot 2 + \ell_1 \cdot (3j+2) + s +1 \ \middle\vert\ 
    \begin{array}{l}
        j,k,\ell_0,\ell_1,s \geq 0,\ s \leq 1,\ j \leq m-2,\\
        3k + \ell_0 + (j+1)\ell_1 + s \leq m-3.
    \end{array}
    \right\}
\end{align*}
are gaps at $P$. 

In the case where $i\leq m-2$, we can argue in a very similar way. In this case, Theorems~\ref{thm:3.12} and~\ref{thm:3.19} ensures that we have functions $f_j\in\mathcal{L}((j+1)D_\infty)$ with valuation $v_P(f_j)=3j+2$ for $j=0,\ldots,i-1$ as well as a function $f_i\in\mathcal{L}((i+1)D_\infty)$ with valuation $3i+3$. 
Then if we define the differentials
\begin{align*}
    \tilde{h}_{k,\ell_0,\ell_1,\ell_2,s}:= F_P^k f_0^{\ell_0} f_j^{\ell_1} f_i^{\ell_2} x_a^s w,
\end{align*}
we get gaps at $v_P(\tilde{h}_{k,\ell_0,\ell_1,\ell_2,s}) +1 $ whenever
\begin{align*}
    v_{P_\infty^k}(\tilde{h}_{k,\ell_0,\ell_1,\ell_2,s}) 
    = -3k - \ell_0 - (j+1)\ell_1 - (i+1)\ell_2 - s + m-3
    \geq 0,
\end{align*}
for all $k=1,\ldots,m$ by Lemma~\ref{lemma:holomorphic_diffs}. 
The valuation of these differentials are 
\begin{align*}
    v_P(\tilde{h}_{k,\ell_0,\ell_1,\ell_2,s}) = 
    \begin{cases}
        k(q+1) + \ell_0\cdot 2 + \ell_1 \cdot (3j+2) + \ell_2\cdot(3i+3) + s & \text{ if $P$ is $\F_{q^2}$-rational,} \\
        kq + \ell_0\cdot 2 + \ell_1 \cdot (3j+2) + \ell_2\cdot(3i+3) + s & \text{ if $P$ is not $\F_{q^2}$-rational.} 
    \end{cases}
\end{align*}
We see that 
\begin{align*}
    \left\{k(q+1) + \ell_0 2 + \ell_1  (3j+2) + \ell_2(3i+3) + s + 1 \ \middle\vert\ 
    \begin{array}{l}
        k,\ell_0,\ell_1,\ell_2, s \geq 0,\ s =0,1,\ j \leq i-1\\
        3k + \ell_0 + \ell_1(j+1)+ \ell_2 (i+1)+ s\leq m-3
    \end{array}
    \right\}
\end{align*}
are gaps when $P$ is rational, and 
\begin{align*}
    \left\{kq + \ell_0\cdot 2 + \ell_1 \cdot (3j+2) + \ell_2\cdot(3i+3) + s + 1 \ \middle\vert\ 
    \begin{array}{l}
        k,\ell_0,\ell_1,\ell_2, s \geq 0,\ s =0,1,\ j \leq i-1\\
        3k + \ell_0 + \ell_1(j+1)+ \ell_2 (i+1)+ s\leq m-3
    \end{array}
    \right\}
\end{align*}
are gaps when $P$ is not rational.
\end{proof}

\subsection{Counting of gaps}

With Lemma \ref{theyaregaps} in place, the final step needed in order to conclude the proof of Theorem \ref{conj:gaps} is to show that each set of gaps indeed contains exactly $g(\mathcal{F}_{(q+1)/3})$ many gaps.
We start with a short lemma that will turn out useful when computing sums later on.
\begin{lemma}\label{floor_sums_cheat_sheet}
    Let $n\in\mathbb{N}$. Then 
 $$
     \sum_{t=0}^{n} \bigg\lfloor \frac{t}{3} \bigg\rfloor = \fl{\frac{n(n-1)}{6}}
     =\begin{cases}
         \frac{n(n-1)}{6}, & \text{if} \ n \equiv 0 \pmod{3}, \\
         \frac{n(n-1)}{6}, & \text{if} \ n \equiv 1 \pmod{3},\\
         \frac{n(n-1)}{6}-\frac{1}{3}, & \text{if} \ n \equiv 2 \pmod{3}.
     \end{cases}
 $$
\end{lemma}
\begin{proof}
It is enough to note that
    \begin{align*}
        \sum_{t=0}^{n} \bigg\lfloor \frac{t}{3} \bigg\rfloor 
        &= \sum_{t=0}^{\fl{\frac{n}{3}}} t + \sum_{t=0}^{\fl{\frac{n-1}{3}}} t + \sum_{t=0}^{\fl{\frac{n-2}{3}}} t\\
        &= \frac{\fl{\frac{n}{3}}\left(\fl{\frac{n}{3}} + 1\right)}{2} + \frac{\fl{\frac{n-1}{3}}\left(\fl{\frac{n-1}{3}} + 1\right)}{2} + \frac{\fl{\frac{n-2}{3}}\left(\fl{\frac{n-2}{3}} + 1\right)}{2}.
    \end{align*}
Now distinguishing the cases $n \equiv 0,1,2 \pmod 3$ the lemma follows directly. %one gets by a direct computation $\fl{n(n-1)/6}$
\end{proof}

We continue with counting the number of gaps for large $\mathcal P$-order. 

\subsubsection{Counting of gaps for large $\mathcal P$-order}
As in Theorem \ref{conj:gaps}, let us denote by $i$ the $\mathcal P$-order of $\alpha$. The aim of this subsection is to prove the following lemma.

\begin{lemma} \label{igtm2}
The sets
  \begin{align*}
        G:=
        \left\{ k(q+1) + \ell_0 \cdot 2 + \ell_1 \cdot (3j+2) + s +1 \ \middle\vert\ 
        \begin{array}{l}
            j,k,\ell_0,\ell_1,s \geq 0,\ s \leq 1,\ j \leq m-2,\\
            3k + \ell_0 + (j+1)\ell_1 + s \leq m-3.
        \end{array}
        \right\}
    \end{align*}
and
    \begin{align*}
        \bar G:=
        \left\{ kq + \ell_0 \cdot 2 + \ell_1 \cdot (3j+2) + s +1 \ \middle\vert\ 
        \begin{array}{l}
            j,k,\ell_0,\ell_1,s \geq 0,\ s \leq 1,\ j \leq m-2,\\
            3k + \ell_0 + (j+1)\ell_1 + s \leq m-3.
        \end{array}
        \right\}
    \end{align*}
both have cardinality equal to $g(\mathcal{F}_{(q+1)/3})$.
    \end{lemma}

    \begin{proof}
For simplicity, we will use the notation $v_{k,\ell_0,\ell_1,s,j} := k(q+1)  + 2 \ell_0 + (3j+2) \ell_1 + s + 1 \in G$ and $\bar{v}_{k,\ell_0,\ell_1,s,j} := kq  + 2 \ell_0 + (3j+2) \ell_1 + s + 1 \in \bar G$, respectively.
Now, the proof will consist of three steps; first we will split the sets $G$ (resp. $\bar{G}$) into three subsets (not necessarily disjoint), then we will compute the intersections of the subsets and finally the size of their union, showing that this is equal to $g(\mathcal{F}_{(q+1)/3})$. Since $G$ and $\bar G$ contain gaps, this will lead to the theorem.
Consider the following subsets of $G$:
\begin{align*}
    G_1 &= \left\{ v_{k,0,0,s,0} \mid s = 0,1,\ k \leq \fl{\frac{m-3-s}{3}} \right\} \\
    G_2 &= \left\{ v_{k,0,1,s,j} \mid s = 0,1,\ j \leq m-4-s,\ k \leq \fl{\frac{m-3-(j+1)-s}{3}} \right\} \\
    G_3 &= \left\{ v_{k,1,1,s,j} \mid s = 0,1,\ \mid j\leq m-5-s,\ k \leq \fl{\frac{m-4-(j+1)-s}{3}} \right\}. 
\end{align*}
Completely analogously, we can define subsets $\bar{G}_1,\ \bar{G}_2,\ \bar{G}_3$ of $\bar{G}$. 
Since the subsets for $\bar G$ are so similar, we will give full details of the proof of Lemma \ref{igtm2} for $G$ and leave it to the reader to check that a completely analogous proof can be carried out for $\bar G$.

We will now compute the cardinalities of the given subsets of $G$.

\textbf{Step 1: The cardinality of $G_1$.} One can easily see that different choices of the parameters $(k,s)$ give rise to different elements in $G_1$ %(resp. $\bar{G}_1$)
. Hence 
\begin{align*}
    |G_1| %= |\bar{G}_1| 
    = \sum_{s=0}^1 \left( \fl{\frac{m-3-s}{3}} + 1\right) 
    = \fl{\frac{m-3}{3}} + \fl{\frac{m-4}{3}} + 2.
    %= \fl{\frac{q+1-9}{9}} + \fl{\frac{q+1-12}{9}} + 2
    %= \fl{\frac{q-8}{9}} + \fl{\frac{q-11}{9}} + 2
\end{align*}

\textbf{Step 2: The cardinality of $G_2$% and $\bar{G}_2$
.}
Like the case for $G_1$%, $\bar{G}_1$
, one can prove by hand that the elements listed in $G_2$ % and $\bar{G}_2$ 
are pairwise distinct. 
Hence using Lemma \ref{floor_sums_cheat_sheet} we get that
\begin{align*}
    |G_2| %= |\bar{G}_2| 
    &= \sum_{s=0}^1\sum_{j=0}^{m-4-s} \left( \fl{\frac{m-3-(j+1)-s}{3}} + 1\right) \\
    %&= \sum_{j=0}^{m-4} \left( \fl{\frac{m-4-j}{3}} + 1\right) + \sum_{j=0}^{m-5} \left( \fl{\frac{m-5-j}{3}} + 1\right) \\
    &= (m-3)+(m-4)+ \sum_{j=0}^{m-4} \fl{\frac{m-4-j}{3}}+ \sum_{j=0}^{m-5} \fl{\frac{m-5-j}{3}} \\
    &= 2m-7 + \sum_{j=0}^{m-4} \fl{\frac{j}{3}}+ \sum_{j=0}^{m-5} \fl{\frac{j}{3}}= 2m-7 + \fl{\frac{m-4}{3}} + 2 \sum_{j=0}^{m-5} \fl{\frac{j}{3}} \\
    &= 2m-7 + \fl{\frac{m-4}{3}} +2\fl{\frac{(m-5)(m-6)}{6}}.
\end{align*}

\textbf{Step 3: The cardinality of $G_3$% and $\bar{G}_3$
.}
Again, one can show that there are no repetitions in $G_3$ % and $\bar{G}_3$
and we may conclude that 
\begin{align*}
    |G_3| %= |\bar{G}_3| 
    &= \sum_{s=0}^{1} \sum_{j=0}^{m-5-s}\left(\fl{\frac{m-5-j-s}{3}}+1\right)\\
    %&= \sum_{j=0}^{m-5}\left(\fl{\frac{m-5-j}{3}}+1\right) + \sum_{j=0}^{m-6}\left(\fl{\frac{m-6-j}{3}}+1\right)\\
    &= (m-4)+(m-5)+ \sum_{j=0}^{m-5}\fl{\frac{m-5-j}{3}}+ \sum_{j=0}^{m-6}\fl{\frac{m-6-j}{3}}\\
    &= 2m-9+ \sum_{j=0}^{m-5}\fl{\frac{j}{3}} + \sum_{j=0}^{m-6}\fl{\frac{j}{3}}= 2m-9 + \fl{\frac{m-5}{3}} + 2 \sum_{j=0}^{m-6} \fl{\frac{j}{3}} \\
    &= 2m-9 + \fl{\frac{m-5}{3}} +2\fl{\frac{(m-6)(m-7)}{6}}.
\end{align*}

\iffalse
last step takes the +1 out and reverses summation order. then we use the trick from the start and get
\begin{align*}
    2\sum_{j=0}^{m-6}\fl{\frac{j}{3}} 
    &= 2\left(\sum_{j=0}^{\fl{\frac{m-6}{3}}}j + \sum_{j=0}^{\fl{\frac{m-7}{3}}}j +\sum_{j=0}^{\fl{\frac{m-8}{3}}}j \right)\\
    &= 2\left( \frac{1}{2}\fl{\frac{m-6}{3}}\left(\fl{\frac{m-6}{3}}+1\right) + \frac{1}{2}\fl{\frac{m-7}{3}}\left(\fl{\frac{m-7}{3}}+1\right) +\frac{1}{2}\fl{\frac{m-8}{3}}\left(\fl{\frac{m-8}{3}}+1\right) \right)\\
    %&= \fl{\frac{q-17}{9}}\left(\fl{\frac{q-17}{9}}+1\right) + \fl{\frac{q-20}{9}}\left(\fl{\frac{q-20}{9}}+1\right) +\fl{\frac{q-23}{9}}\left(\fl{\frac{q-23}{9}}+1\right)
\end{align*}
giving us 
\begin{align*}
    |G_3| %= |\bar{G}_3| 
    &= 2m-9+ \fl{\frac{m-5}{3}} + \fl{\frac{m-6}{3}}\left(\fl{\frac{m-6}{3}}+1\right) + \fl{\frac{m-7}{3}}\left(\fl{\frac{m-7}{3}}+1\right) +\fl{\frac{m-8}{3}}\left(\fl{\frac{m-8}{3}}+1\right)
\end{align*}
\fi

\textbf{Step 4: Intersections of the subsets $G_i$} Our aim is to show that $G_1 \cap G_2  = G_1 \cap G_3  = \emptyset$ and compute the cardinality of $G_2 \cap G_3$.
In order to show that $G_1 \cap G_2 = \emptyset$, assume contrarily that there are elements $v_{k,0,0,s,0} \in G_1$ and $v_{\tilde{k},0,1,\tilde{s},\tilde{j}} \in G_2$ such that
\begin{align*}
    k(q+1) + s + 1 = \tilde{k}(q+1) + (3\tilde{j}+2) + \tilde{s} + 1.
\end{align*}
Then $s\equiv \tilde{s} +2\m{3}$ implying that $ s\equiv \tilde{s} -1\m{3}$, in particular $(s,\tilde{s})=(0,1)$. Thus
\begin{align*}
    k(q+1) = \tilde{k}(q+1) + 3\tilde{j}+3
    \iff km = \tilde{k}m + \tilde{j} + 1,
\end{align*}
and we see that $\tilde{j}+1\equiv 0\m{m}$. However, $\tilde{j} +1\leq m-4-\tilde{s}+1 = m-4$ is thus a contradiction, and we conclude that $G_1 \cap G_2 = \emptyset$. 
The strategy to show that $G_1 \cap G_3 = \emptyset$ is very similar and therefore will be omitted.
What is left is to compute $G_2 \cap G_3$ in order to compute the total sum.
Let us therefore first assume the existence of two elements $v_{k,0,1,s,j}\in G_2$ and $v_{\tilde{k},1,1,\tilde{s},\tilde{j}} \in G_3$ such that
\begin{align*}
    k(q+1) + 3j + 2 + s + 1 = \tilde{k}(q+1) + 2 + 3\tilde{j} + 2 + \tilde{s} + 1.
\end{align*}
We see that  $s \equiv 2 +\tilde{s} \m{3}$, so $(s,\tilde{s})=(0,1)$. Then
\begin{align*}
     k(q+1) + 3j= \tilde{k}(q+1) + 3\tilde{j} + 3
     \iff km + j= \tilde{k}m + \tilde{j} + 1,
\end{align*}
so $j \equiv \tilde{j}+1\m{m}$, in fact $j=\tilde{j}+1$, since $j,\tilde{j}<m$.
This implies that $k=\tilde{k}$, so we can obtain a new bound, namely $k,\tilde{k} \leq \fl{\frac{m-3-(\tilde{j}+1)-\tilde{s}-1}{3}} = \fl{\frac{m-5-(\tilde{j}+1)}{3}} = \fl{\frac{m-5-j}{3}}$. We can also obtain a new bound for $j$, namely $1 \leq \tilde{j}+1 = j \leq m-5-\tilde{s}+1=m-5$.
So the elements in $G_2 \cap G_3$ must be the elements in $G_2$ that satisfies these new bounds, i.e.
\begin{align*}
    G_2 \cap G_3 = \left\{ v_{k,0,1,0,j} \mid j = 1,\ldots, m-5,\ k = 0, \ldots, \fl{\frac{m-5-j}{3}} \right\}.
\end{align*}

%We can compute $\bar{G}_2 \cap \bar{G}_3$ quite similarly, so let $\bar{v}_{k,0,1,s,j}\in \bar{G}_2$ and $\bar{v}_{\tilde{k},1,1,\tilde{s},\tilde{j}} \in \bar{G}_3$ such that
%\begin{align*}
%    kq + 3j + 2 + s + 1 = \tilde{k}q + 2 + 3\tilde{j} + 2 + \tilde{s} + 1.
%\end{align*}
%Arguing as before, we see that $3j + s = 2 + 3\tilde{j} + \tilde{s}$, so $(s,\tilde{s})=(0,1)$. This implies that $kq + 3j = \tilde{k}q + 3\tilde{j} + 3$, so $3j \equiv 3\tilde{j} + 3 \m{q}$, in fact $j =\tilde{j}+1$. From here on the argument is the same as in the rational case, and we can compute both $|G_2 \cap G_3|$ and $|\bar{G}_2 \cap \bar{G}_3|$ in one go. 
Note that since $(G_2 \cap G_3) \subseteq G_2$ %and $(\bar{G}_2 \cap \bar{G}_3) \subseteq \bar{G}_2$
, there cannot be repetitions in either of the two sets, so we can compute the size as
\begin{align*}
    |G_2 \cap G_3| &= \sum_{j=1}^{m-5} \left(\fl{\frac{m-5-j}{3}} + 1\right)
    = \sum_{j=0}^{m-6} \fl{\frac{j}{3}} + (m-5)\\
    &=m-5+\fl{\frac{(m-6)(m-7)}{6}}.
    %= \sum_{j=0}^{\fl{\frac{m-6}{3}}} j + \sum_{j=0}^{\fl{\frac{m-7}{3}}} j + \sum_{j=0}^{\fl{\frac{m-8}{3}}} j  + (m-5)\\
    %&= \frac{1}{2} \fl{\frac{m-6}{3}} \left(\fl{\frac{m-6}{3}}+1\right) + \frac{1}{2} \fl{\frac{m-7}{3}} \left(\fl{\frac{m-7}{3}}+1\right) + \frac{1}{2} \fl{\frac{m-8}{3}} \left(\fl{\frac{m-8}{3}}+1\right) + (m-5)\\
    %&= \frac{1}{2} \fl{\frac{q-17}{9}} \left(\fl{\frac{q-17}{9}}+1\right) + \frac{1}{2} \fl{\frac{q-20}{9}} \left(\fl{\frac{q-20}{9}}+1\right) + \frac{1}{2} \fl{\frac{q-23}{9}} \left(\fl{\frac{q-23}{9}}+1\right) j + \frac{q-14}{3}.
\end{align*}
Now we are ready to complete our proof by adding all the contributions together.

\textbf{Step 5: The cardinality of $G_1 \cup G_2 \cup G_3$.}
As discussed at the beginning of the proof, we wish to show that $|G_1 \cup G_2 \cup G_3|=g(F_m)$.

\begin{align*}
    |G_1 \cup G_2 \cup G_3| &= |G_1| + |G_2| + |G_3| - |G_2 \cap G_3| \\
    &=\fl{\frac{m-3}{3}} + \fl{\frac{m-4}{3}} + 2+2m-7 + \fl{\frac{m-4}{3}} +2\fl{\frac{(m-5)(m-6)}{6}}\\
    &+ 2m-9 + \fl{\frac{m-5}{3}} +2\fl{\frac{(m-6)(m-7)}{6}}-\bigg( m-5+\fl{\frac{(m-6)(m-7)}{6}}\bigg)\\
    &=4-m+\fl{\frac{m}{3}}+2\fl{\frac{m-1}{3}}+\fl{\frac{m-2}{3}}+2\fl{\frac{m(m-5)}{6}}+\fl{\frac{m(m-1)}{6}}\\
    &=\begin{cases} 4-m+\frac{m}{3}+3\frac{m-3}{3}+2\frac{m(m-5)}{6}+\frac{m(m-1)}{6}, & \text{if} \ m \equiv 0 \pmod 3, \\
    4-m+3\frac{m-1}{3}+\frac{m-4}{3}+2\frac{m^2-5m-2}{6}+\frac{m(m-1)}{6}, & \text{if} \ m \equiv 1 \pmod 3,\\
4-m+4\frac{m-2}{3}+2\frac{m(m-5)}{6}+\frac{m^2-m-2}{6}, & \text{if} \ m \equiv 2 \pmod 3\end{cases}\\
&%=\frac{3m^2-9m+6}{6}
=\frac{(m-1)(m-2)}{2}=g(\mathcal{F}_m).
\end{align*}
We can now conclude that $|G| = |G_1 \cup G_2 \cup G_3| = g(\mathcal{F}_{(q+1)/3})$. 
\end{proof}

\subsubsection{Counting of gaps for small $\mathcal P$-order}\label{counting_gaps_for_small_P_orders}

We will now proceed with the proof of Theorem \ref{conj:gaps} in case $m \geq i+2$. Our aim is to prove the following lemma.

\begin{lemma} \label{counting-ismall}
  Let $i \leq m-2$, then the sets   
    \begin{align*}
        G:=
        \left\{  k(q+1) + \ell_0\cdot 2 + \ell_1 \cdot (3j+2) + \ell_2\cdot(3i+3) + s + 1\
        \middle\vert\ k,\ell_0,\ell_1,\ell_2, s \geq 0,\ s =0,1, \right.\\
        \left.  \text{and }\ j \leq i-1,\ 3k + \ell_0 + \ell_1(j+1)+ \ell_2 (i+1)+ s\leq m-3
         \right\}
    \end{align*}
 and 
    \begin{align*}
        \bar G
        :=
        \left\{  kq + \ell_0\cdot 2 + \ell_1 \cdot (3j+2) + \ell_2\cdot(3i+3) + s + 1
        \middle\vert\ k,\ell_0,\ell_1,\ell_2, s \geq 0,\ s =0,1,\ j \leq i-1,\right. \\
        \left. \text{and }\ 3k + \ell_0 + \ell_1(j+1)+ \ell_2 (i+1)+ s\leq m-3
        \right\}
    \end{align*}
both have cardinality equal to $g(\mathcal{F}_{(q+1)/3})$.
 \end{lemma}

As in the previous, we divide $G$ into subsets and compute their sizes and intersections. 
Denote by $v_{k,\ell_0,\ell_1,\ell_2, s, j}=k(q+1) + \ell_0\cdot 2 + \ell_1 \cdot (3j+2) + \ell_2\cdot(3i+3) + s + 1$, then we define the following subsets:
\begin{align*}
    G_{1a} &= \left\{v_{k,0,0,0, s, 0} \middle\vert\ 
        \begin{array}{l}
            0\leq s \leq 1,\quad 0\leq k, \text{ and } \\
            3k + s \leq m-3.
        \end{array}
        \right\},
\\[2ex]
    G_{2a} &= \left\{v_{k,0,1,0, s, j} \middle\vert\ 
        \begin{array}{l}
            0\leq s \leq 1,\quad 0 \leq \ j \leq i-1, \quad 0\leq k, \text{ and } \\
            3k + (j+1) + s \leq m-3.
        \end{array}
        \right\},
\\[2ex]
    G_{3a} &= \left\{v_{k,1,1,0, s, j} \middle\vert\ 
        \begin{array}{l}
            0\leq s \leq 1, \quad 0 \leq \ j \leq i-1, \quad 0\leq k, \text{ and } \\
            3k + (j+1) + s \leq m-4.
        \end{array}
        \right\},
\\[2ex]
    G_{1b} &= \left\{v_{k,0,0,\ell_2, s, 0} \middle\vert\ 
        \begin{array}{l}
            0\leq s \leq 1, \quad 0\leq k, \quad 1\leq \ell_2, \text{ and } \\
            3k + (i+1)\ell_2 + s \leq m-3.
        \end{array}
        \right\},
\\[2ex]
    G_{2b} &= \left\{v_{k,0,1,\ell_2, s, j} \middle\vert\ 
        \begin{array}{l}
            0\leq s \leq 1, \quad 0 \leq \ j \leq i-1, \quad 0\leq k, \quad 1\leq \ell_2, \text{ and } \\
            3k + (j+1) + (i+1)\ell_2 + s \leq m-3.
        \end{array}
        \right\},
\\[2ex]
    G_{3b} &= \left\{v_{k,1,1,\ell_2, s, j} \middle\vert\ 
        \begin{array}{l}
            0\leq s \leq 1, \quad 0 \leq \ j \leq i-1, \quad 0\leq k,\quad 1 \leq \ell_2, \text{ and } \\
            3k + (j+1) + (i+1)\ell_2 + s \leq m-4.
        \end{array}
        \right\}.
\end{align*}
\begin{remark}\label{rem:GandGbar}
    All the definitions above can be modified to obtain analogue subsets of $\bar G$ (similar to the situation in Lemma \ref{igtm2}), so let $\bar{G}_{xa},\ \bar{G}_{xb}$ for $x=1,2,3$ be defined in the obvious way.
The strategy at this point will be the same as for Lemma \ref{igtm2}. Noting that that the subsets presented above (resp. their analogues in $\bar G$) do not contain repetitions for different choices of the parameters $k,s,\ell_1\ell_2$ and $j$, computing their cardinalities as well as the cardinality of their mutual intersections can be done by counting the number of choices of these parameters. This implies that $|G_{xa}|=|\bar{G}_{xa}|$ and $|G_{xb}|=|\bar{G}_{xb}|$ for $x=1,2,3$, and similar statements hold for the cardinalities of their mutual intersections. We will therefore from now on only state and prove the cardinalities of $G_{xa}, G_{xb}$,  and several of their intersections, but we will in fact then also have proven analogous formulas for the cardinalities of the sets $\bar{G}_{xa}, \bar{G}_{xb}$ and their intersections.
\end{remark}

Remark \ref{rem:GandGbar} implies that we only need to prove Lemma \ref{counting-ismall} for the set $G$, the case $\bar{G}$ being very similar. We follow the same approach as in the proof of Lemma \ref{igtm2}. Computing the cardinalities $|G_{xa}|$ %(and hence also the cardinalities $|\bar{G}_{xa}|$) 
will be very similar to what we did in the proof of Lemma \ref{igtm2}. However, the cardinalities $|G_{xb}|$ %(and hence also $|\bar{G}_{xb}|$) 
depend more heavily on the relation between $m$ and $i$ and will therefore be computed on a case by case basis. 

\textbf{Step 1: The cardinality $|G_{xa}|$.}

Since $G_{1a}$ does not depend on $j$, we see that 
$$|G_{1a}| = |G_1|=\fl{\frac{m-3}{3}} + \fl{\frac{m-4}{3}} + 2.$$
For the rest of the sets (e.g. $x=2,3$), the bound on $j$ has an influence. Note that since in $G_{2a}$ we need that $k \geq 0$ and $3k + (j+1) + s \leq m-3$, we see that $3k + (j+1) \le m-3$ or in other words: $j \le m-4-s$. Since also $j \leq i-1$, we see that $j \le \min\{i-1,m-4-s\}$. All in all we get:
%Thus whenever $m-4 < i$ (i.e. $m=i+2$ or $m=i+3$), we have $|G_{2a}| = |\bar{G}_{2a}| = |G_2|$, whilst if $m-4 \geq i$, we get
%\begin{align*}
%    |G_{2a}| = |\bar{G}_{2a}| 
%    &= \sum_{j=0}^{i-1} \fl{\frac{m-(j+1)}{3}}  + \sum_{j=0}^{i-1} \fl{\frac{m-(j+2)}{3}} \\
%    &= \sum_{j=1}^{i} \fl{\frac{m-j}{3}}  + \sum_{j=2}^{i+1} \fl{\frac{m-j}{3}} \\
%    &= 2\sum_{j=1}^{i} \fl{\frac{m-j}{3}}  - \fl{\frac{m-1}{3}} + \fl{\frac{m-(i+1)}{3}}.
%\end{align*}
\begin{align*}
    |G_{2a}| 
    &= \sum_{s=0}^1\sum_{j=0}^{\min\{i-1,m-4-s\}} \left( \fl{\frac{m-4-j-s}{3}} + 1\right)\\
&=\min\{i-1,m-4\}+\min\{i-1,m-5\}+2+\sum_{s=0}^1\sum_{j=0}^{\min\{i-1,m-4-s\}} \fl{\frac{m-4-j-s}{3}}\\
%&=\min\{i-1,m-4\}+\min\{i-1,m-5\}+2+\\
%&+\sum_{j=0}^{\min\{i-1,m-4\}} \fl{\frac{m-4-j}{3}}+\sum_{j=0}^{\min\{i-1,m-5\}} \fl{\frac{m-5-j}{3}}\\
&=\min\{i-1,m-4\}+\min\{i-1,m-5\}+2+\hspace{-7ex}\sum_{t=m-4-\min\{i-1,m-4\}}^{m-4} \fl{\frac{t}{3}}+\hspace{-2.5ex}\sum_{t=m-5-\min\{i-1,m-5\}}^{m-5} \fl{\frac{t}{3}}\\
&=\begin{cases} 2i+\sum_{t=m-3-i}^{m-4} \fl{\frac{t}{3}}+\sum_{t=m-4-i}^{m-5} \fl{\frac{t}{3}}, & \hspace{-1.5ex}\text{if} \ i \leq m-4\\
2m-7+\sum_{t=0}^{m-4} \fl{\frac{t}{3}}+\sum_{t=0}^{m-5} \fl{\frac{t}{3}}, & \hspace{-1.5ex}\text{if} \ i  \ge m-3\end{cases}\\
%&=\begin{cases} 2i+\fl{\frac{m-4}{3}}+\fl{\frac{m-4-i}{3}}+2\fl{\frac{(m-5)(m-6)}{6}}-2\fl{\frac{(m-4-i)(m-5-i)}{6}}, & \text{if} \ i \leq m-4\\
%2m-7+\fl{\frac{m-4}{3}}+2\fl{\frac{(m-5)(m-6)}{6}}, & \text{if} \ i \in \{m-3,m-2\}\end{cases}\\
&=\begin{cases} 2i+\fl{\frac{m-4}{3}}+\fl{\frac{m-4-i}{3}}+2\fl{\frac{(m-5)(m-6)}{6}}-2\fl{\frac{(m-4-i)(m-5-i)}{6}}, & \hspace{-2ex}\text{if} \ i \leq m-4\\
|G_2|, & \hspace{-2ex}\text{if} \ i \ge m-3.
\end{cases}\\
&=\begin{cases} |G_2|-2m+7+2i+\fl{\frac{m-4-i}{3}}-2\fl{\frac{(m-4-i)(m-5-i)}{6}}, & \hspace{-2ex}\text{if} \ i \leq m-4\\
|G_2|, & \hspace{-2ex}\text{if} \ i \ge m-3.
\end{cases}\\
\end{align*}
Note that since we are assuming thoughout at this point that $i \le m-2$, the second case only occurs for $i \in \{m-3,m-2\}$.
Similarly

\begin{align*}
|G_{3a}|  
   & = \sum_{s=0}^1\sum_{j=0}^{\min\{i-1,m-5-s\}} \left( \fl{\frac{m-5-j-s}{3}} + 1\right)\\
&=\begin{cases} 2i+\fl{\frac{m-5}{3}}+\fl{\frac{m-5-i}{3}}+2\fl{\frac{(m-6)(m-7)}{6}}-2\fl{\frac{(m-5-i)(m-6-i)}{6}}, & \text{if} \  i \leq m-5\\
|G_3|, & \text{if}  \ i \ge m-4.
\end{cases}\\
&=\begin{cases} |G_3|-2m+9+2i+\fl{\frac{m-5-i}{3}}-2\fl{\frac{(m-5-i)(m-6-i)}{6}}, & \text{if} \  i \leq m-5\\
|G_3|, & \text{if}  \ i \ge m-4.
\end{cases}
\end{align*}

%\begin{align*}
%    |G_{3a}| & 
%    = \sum_{s=0}^1\sum_{j=0}^{\min\{i-1,m-5-s\}} \left( \fl{\frac{m-5-j-s}{3}} + 1\right)  \\
%&=\min\{i-1,m-5\}+\min\{i-1,m-6\}+2+\\
%&+\sum_{j=0}^{\min\{i-1,m-5\}} \fl{\frac{m-5-j}{3}}+\sum_{j=0}^{\min\{i-1,m-6\}} \fl{\frac{m-6-j}{3}}\\
%&=\min\{i-1,m-5\}+\min\{i-1,m-6\}+2+\\
%&+\sum_{t=m-5-\min\{i-1,m-5\}}^{m-5} \fl{\frac{t}{3}}+\sum_{t=m-6-\min\{i-1,m-6\}}^{m-6} \fl{\frac{t}{3}}\\
%&=\begin{cases} 2i+\sum_{t=m-4-i}^{m-5} \fl{\frac{t}{3}}+\sum_{t=m-5-i}^{m-6} \fl{\frac{t}{3}}, \ if \ i \leq m-5\\
%2m-9+\sum_{t=0}^{m-5} \fl{\frac{t}{3}}+\sum_{t=0}^{m-6} \fl{\frac{t}{3}}, \ if \ i \geq m-4\end{cases}\\
%&=\begin{cases} 2i+\fl{\frac{m-5}{3}}+\fl{\frac{m-5-i}{3}}+2\fl{\frac{(m-6)(m-7)}{6}}-2\fl{\frac{(m-5-i)(m-6-i)}{6}}, \ if \ i \leq m-5\\
%2m-9+\fl{\frac{m-5}{3}}+2\fl{\frac{(m-6)(m-7)}{6}}, \ if \ i \in \{m-4,m-3,m-2\}.\end{cases}\\
%&=\begin{cases} 2i+\fl{\frac{m-5}{3}}+\fl{\frac{m-5-i}{3}}+2\fl{\frac{(m-6)(m-7)}{6}}-2\fl{\frac{(m-5-i)(m-6-i)}{6}}, & \hspace{-2ex}\text{if} \  i \leq m-5\\
%|G_3|, & \hspace{-2ex}\text{if}  \ i \in \{m-4,m-3,m-2\}.\end{cases} \\
%\end{align*}
%We will come back to these expressions later. 

\textbf{Step 2: The cardinality of $G_{xb}$.}

When computing the cardinalities of $G_{xb}$ a particular type of summation will occur all the time, as we will see shortly. For this reason, we introduce the following notation:

$$S(j):=\sum_{\ell_2=1}^{\fl{\frac{m-j}{i+1}}}\left(\fl{\frac{m-j-(i+1)\ell_2}{3}}+1\right).$$
Starting with $G_{1b}$ we see that for a given $s=0,1$ then $\ell_2 \leq \fl{\frac{m-3-s}{i+1}}$ is necessary, while by construction $\ell_2 \geq 1$. This implies that for $i >m-4-s$ we would have no contribution in $|G_{1b}|$ (the two aforementioned bounds contradict each other). This implies that
\\

\begin{align*}|G_{1b}| & = \begin{cases} \sum_{s=0}^1 \sum_{\ell_2=1}^{\fl{\frac{m-3-s}{i+1}}}\bigg(\fl{\frac{m-3-s-(i+1)\ell_2}{3}}+1\bigg), & \text{if} \ i \leq m-5,\\
1, & \text{if} \ i=m-4,\\
0, & \text{if} \ i \geq m-3\end{cases}\\
& = \begin{cases} S(3)+S(4), & \text{if} \ i \leq m-5,\\
1, & \text{if} \ i=m-4,\\
0, & \text{if} \ i \ge m-3.
\end{cases}
\end{align*}

%\noindent
%$=\begin{cases} \fl{\frac{m-3}{i+1}}+\fl{\frac{m-4}{i+1}}+\sum_{\ell_2=1}^{\fl{\frac{m-3}{i+1}}}\fl{\frac{m-3-(i+1)\ell_2}{3}}+\sum_{\ell_2=1}^{\fl{\frac{m-4}{i+1}}}\fl{\frac{m-4-(i+1)\ell_2}{3}}, & \text{if} \ i \leq m-5,\\
%1, & \text{if} \ i=m-4,\\
%0, & \text{if} \ i \ge m-3.
%\end{cases}$
%\\

We now consider $G_{2b}$. We see that for a given $s=0,1$ then $j \leq m-4-s$ is necessary, while by construction $j \geq 1$. This implies that $0 \leq j \leq \min\{i-1,m-4-s\}$. Similarly when $j$ is fixed, then $\ell_2=1,\ldots,\fl{\frac{m-4-s-j}{i+1}}$ which makes sense only if $m-4-s-j \geq i+1$, that is, $j \leq m-5-s-i$. This forces $0 \leq j \leq \min\{i-1,m-4-s,m-5-s-i\}=\min\{i-1,m-5-s-i\}$. If $i \geq m-4-s$ then this minimum is negative, which is not possible. Hence for a fixed $s=0,1$ we possibly have a contribution in $|G_{2b}|$ only if $i \leq m-5-s$. If this is the case, and $j=0,\ldots, \min\{i-1,m-5-s-i\}$ is fixed then necessarily $k=0,\ldots,\fl{\frac{m-4-s-j-\ell_2(i+1)}{3}}$. This gives

$$|G_{2b}|=\begin{cases} \sum_{s=0}^1 \sum_{j=0}^{\min\{i-1,m-5-s-i\}}S(j+4+s)
%\sum_{\ell_2=1}^{\fl{\frac{m-4-s-j}{i+1}}}\bigg(\fl{\frac{m-4-s-j-(i+1)\ell_2}{3}}+1\bigg)
, & \text{if} \ i \leq m-6,\\
1, & \text{if} \ i=m-5,\\
0, & \text{if} \ i\ge m-4.\\
%\in \{m-2,m-3,m-4\}.
\end{cases}$$

Finally we consider $G_{3b}$. Arguing as in the previous case we see that for a given $s=0,1$ then $0 \leq j \leq \min\{i-1,m-6-s-i\}$ and $\ell_2=1,\ldots,\fl{\frac{m-7-s-j}{i+1}}$. If $i \geq m-6-s$ then $\min\{i-1,m-6-s-i\}=m-6-s-i <0$, which is not possible. Hence for a fixed $s=0,1$ we possibly have a contribution in $|G_{3b}|$ only if $i \leq m-6-s$. If this is the case, and $j=0,\ldots, \min\{i-1,m-6-s-i\}$ is fixed then necessarily $k=0,\ldots,\fl{\frac{m-5-s-j-\ell_2(i+1)}{3}}$. This gives

$$|G_{3b}|=\begin{cases} \sum_{s=0}^1 \sum_{j=0}^{\min\{i-1,m-6-s-i\}} S(j+5+s)
%\sum_{\ell_2=1}^{\fl{\frac{m-5-s-j}{i+1}}}\bigg(\fl{\frac{m-5-s-j-(i+1)\ell_2}{3}}+1\bigg)
, & \text{if} \ i \leq m-7,\\
1, & \text{if} \ i=m-6,\\
0, & \text{if} \ i \ge m-5.\\ 
%\in \{m-2,m-3,m-4,m-5\}
\end{cases}$$

\textbf{Step 3: Cardinality of $\left( \bigcup_{x=1,2,3} G_{xa} \right) \cap \left( \bigcup_{x=1,2,3} G_{xb} \right)$}
The first intersection we will compute is $\left( \bigcup_{x=1,2,3} G_{xa} \right) \cap \left( \bigcup_{x=1,2,3} G_{xb} \right)$, in particular our goal will be to show that this intersection is in fact empty. 
Assume therefore contrarily that there exists two elements $v_{k,\ell_0,\ell_1,0,s,j}\in\left( \bigcup_{x} G_{xa} \right)$ and $v_{\Tilde{k},\Tilde{\ell_0},\Tilde{\ell_1},\ell_2,\Tilde{s},\Tilde{j}}\in \left( \bigcup_x G_{xb} \right)$ with $\ell_2\geq 1$ such that 
\begin{align*}
    k(q+1) + \ell_0\cdot 2 + \ell_1 \cdot (3j+2) + s + 1 &= \Tilde{k}(q+1) + \Tilde{\ell_0}\cdot 2 + \Tilde{\ell_1} \cdot (3\Tilde{j}+2) +\ell_2 (3i+3) + \Tilde{s} + 1.
\end{align*}
If we consider this modulo $q+1$, we get that 
\begin{align*}
    \ell_0\cdot 2 + \ell_1 \cdot (3j+2) + s \equiv \Tilde{\ell_0}\cdot 2 + \Tilde{\ell_1} \cdot (3\Tilde{j}+2) +\ell_2 (3i+3) + \Tilde{s}  \mod (q+1).
\end{align*}
Since both sides are smaller than $q+1$ we get
\begin{align*}
    \ell_0\cdot 2 + \ell_1 \cdot (3j+2) + s + 1 = \Tilde{\ell_0}\cdot 2 + \Tilde{\ell_1} \cdot (3\Tilde{j}+2) +\ell_2 (3i+3) + \Tilde{s} + 1.
\end{align*}
Now, this is what we need to reach a contradiction. Consider first the left side
\begin{align*}
    \ell_0\cdot 2 + \ell_1 \cdot (3j+2) + s + 1 \leq 2 + 3j+2 + 1 + 1 = 3j + 6 \leq 3(i-1)+6 = 3i+3.
\end{align*}
On the other hand, since $\ell_2 \geq 1$, we have that
\begin{align*}
    \Tilde{\ell_0}\cdot 2 + \Tilde{\ell_1} \cdot (3\Tilde{j}+2) +\ell_2 (3i+3) + \Tilde{s} +1 > 3i+3,
\end{align*}
a contradiction. We can therefore conclude that there cannot exist an element in $\left( \bigcup_{x=1,2,3} G_{xa} \right) \cap \left( \bigcup_{x=1,2,3} G_{xb} \right)$, and it is in fact the empty set. 
The argument for showing that %$\left( \bigcup_{x=1,2,3} \bar G_{xa} \right) \cap \left( \bigcup_{x=1,2,3} \bar G_{xb} \right)=\emptyset$, 
$G_{1a}\cap G_{2a} = G_{1a}\cap G_{3a} = \emptyset$ and $G_{1b}\cap G_{2b}= G_{1b}\cap G_{3b} = \emptyset$ is analogous to the above, and is omitted. 

Furthermore, the argument used when computing $G_{2} \cap G_{3}$ can be reused to finding $G_{2a} \cap G_{3a}$, only with the addition of a stricter bound on $j$. Hence
$$
    |G_{2a} \cap G_{3a}| = 
    \sum_{j=1}^{\min\{i-1,m-5\}} \left(\fl{\frac{m-5-j}{3}} +1 \right)
$$
\begin{align*}
 = & \min\{i-1,m-5\}+\fl{\frac{(m-6)(m-7)}{6}}\\
 &-\fl{\frac{(m-6-\min\{i-1,m-5\})(m-7-\min\{i-1,m-5\})}{6}}\\
 =& |G_2 \cap G_3|-m+5+\min\{i-1,m-5\}\\   
  &-\fl{\frac{(m-6-\min\{i-1,m-5\})(m-7-\min\{i-1,m-5\})}{6}}.\\
\end{align*}

The last intersection we will consider is $G_{2b}\cap G_{3b}$, so let $v_{k,0,1,\ell_2,s,j}\in G_{2b}$ and $v_{\Tilde{k},1,1,\Tilde{\ell_2},\Tilde{s},\Tilde{j}}\in G_{3b}$ with $\ell_2\geq 1$ such that
\begin{align*}
    k(q+1) + (3j+2) + \ell_2 (3i+3)+ s + 1 = \Tilde{k}(q+1) + 2 + (3\Tilde{j}+2) + \Tilde{\ell}_2 (3i+3) + \Tilde{s} + 1.
\end{align*}
Again, we consider this modulo $q+1$ and use the pole bound to arrive at 
\begin{align*}
    3j + \ell_2 (3i+3)+ s = 2 + 3\Tilde{j} + \Tilde{\ell}_2 (3i+3) + \Tilde{s}.
\end{align*}
We may then get $s\equiv 2+\Tilde{s}\mod3$, so $(s,\Tilde{s})=(0,1)$, and we get
\begin{align*}
    3j + \ell_2 (3i+3) = 3 + 3\Tilde{j} + \Tilde{\ell}_2 (3i+3)
    \iff 
    j + \ell_2 (i+1) = 1 + \Tilde{j} + \Tilde{\ell}_2 (i+1).
\end{align*}
Thus $j \equiv 1 + \Tilde{j} \mod (i+1)$, in particular $j = 1 + \Tilde{j}$, since $j, \Tilde{j}\leq i-1$. This implies that $\ell_2 (i+1) =\Tilde{\ell}_2 (i+1)$, so $\ell_2=\Tilde{\ell}_2$, which in turn implies that $k = \Tilde{k}$.
We conclude that
\begin{align*}
    |G_{2b} \cap G_{3b}| &= 
    \sum_{j=0}^{\min \{i-2,m-7-i\}} \sum_{\ell_2=1}^{\fl{\frac{m-6-j}{i+1}}}\left(\fl{\frac{m-6-j-\ell_2(i+1)}{3}} +1 \right)\\
    &=\sum_{j=0}^{\min \{i-2,m-7-i\}} S(j+6).
\end{align*}
Note that $|G_{2b} \cap G_{3b}|=0$ if $i \geq m-6$. 

Putting everything together we get, we are now ready to compute
\begin{equation}\label{eq:cardGxaGyb}
|\cup_{x=1,2,3} (G_{xa} \cup G_{xb})|=\sum_{x=1}^3|G_{xa}|+\sum_{x=1}^3|G_{xb}|-|G_{2a} \cap G_{3a}|-|G_{2b} \cap G_{3b}|.
\end{equation}
As before, the goal is to show that the right-hand side in Equation \eqref{eq:cardGxaGyb} gives $g(\mathcal{F}_m)$, that is to say: $(m-1)(m-2)/2$. We will use the expressions for each of the terms in the right-hand side in Equation \eqref{eq:cardGxaGyb} that we obtained previously. 

Recall that we are assuming at this point that $i \le m-2$. To find manageable expressions while evaluating the right-hand side in Equation \eqref{eq:cardGxaGyb}, we will distinguish some cases. An overview of these cases is as follows:
\begin{itemize}
    \item[] {\bf Case 1:} $m-6 \le i \le m-2$.
    \item[] {\bf Case 2:} $m-7 \ge i$ and $2i > m-5$.
    \item[] {\bf Case 3:} $m-7 \ge i$ and $2i = m-5$.
    \item[] {\bf Case 4:} $m-7 \ge i$ and $2i< m-5$.
\end{itemize}
Now we address these cases one after the other.

\noindent {\bf Case 1:} $m-6 \le i \le m-2$.

$\bullet$ If $i \geq m-3$ (that is $i=m-3$ or $i=m-2$) then $|G_{xb}|=0$ for all $x=1,2,3$ and so is $|G_{2b} \cap G_{3b}|$. Since one also has $|G_{xa}|=|G_x|$ for all $x=1,2,3$ and $|G_{2a} \cap G_{3a}|=|G_2 \cap G_3|$, we get
$$|\cup_{x=1,2,3} (G_{xa} \cup G_{xb})|=\sum_{x=1}^3|G_{xa}|-|G_{2a} \cap G_{3a}|=|G_1 \cup G_2 \cup G_3|=g(\mathcal{F}_m),$$
from the computation in the previous section.
    
$\bullet$ If $i=m-4$ then $|G_{1a}|=|G_1|$, $|G_{2a}|=|G_2|-1$, $|G_{3a}|=|G_3|$, $|G_{2a} \cap G_{3a}|=|G_2 \cap G_3|$, $|G_{1b}|=1$ and $G_{2b}=G_{3b}=\emptyset$, hence $|\cup_{x=1,2,3} (G_{xa} \cup G_{xb})|=g(\mathcal{F}_m)$ follows again from the previous section.

$\bullet$ If $i=m-5$ then $|G_{1a}|=|G_1|$, $|G_{2a}|=|G_2|-3$, $|G_{3a}|=|G_3|-1$, $|G_{2a} \cap G_{3a}|=|G_2 \cap G_3|-1$, $|G_{1b}|=2$, $|G_{2b}|=1$ and $|G_{3b}|=0$ and $G_{2b} \cap G_{3b}=\emptyset$, hence 
$$|\cup_{x=1,2,3} (G_{xa} \cup G_{xb})|=|G_1|+|G_2|-3+|G_3|-1+2+1+0-|G_2 \cap G_3|+1=|G_1 \cup G_2 \cup G_3|=g(\mathcal{F}_m).$$

$\bullet$ If $i=m-6$ and $i \ge 2$ then $|G_{1a}|=|G_1|$, $|G_{2a}|=|G_2|-5$, $|G_{3a}|=|G_3|-3$, $|G_{2a} \cap G_{3a}|=|G_2 \cap G_3|-2$, $|G_{1b}|=2$, $|G_{2b}|=3$, $|G_{3b}|=1$ and $G_{2b} \cap G_{3b}=\emptyset$. If $i=m-6$ and $i=1$, then $|G_{1b}|=3$ and  $|G_{2b}|=2$, but otherwise all cardinalities are the same as in the case $i=m-6$, $i \ge 2$. We obtain
$$|\cup_{x=1,2,3} (G_{xa} \cup G_{xb})|=|G_1|+|G_2|-5+|G_3|-3+2+3+1-|G_2 \cap G_3|+2=|G_1 \cup G_2 \cup G_3|=g(\mathcal{F}_m).$$

\noindent {\bf Case 2:} $m-7 \ge i$ and $2i > m-5$.

Since we already have shown previously that $|G_1|+|G_2|+|G_3|-|G_2 \cap G_3|=g(\mathcal{F}_m)$, we obtain from the cardinalities computed in Step 1-3 in Subsection~\ref{counting_gaps_for_small_P_orders} that
$$|\cup_{x=1,2,3} (G_{xa} \cup G_{xb})|=g(\mathcal{F}_m)+\Delta,$$
where
\begin{align*}
\Delta =&-3m+14+3i+2\fl{\frac{m-4-i}{3}}-2\fl{\frac{(m-4-i)(m-5-i)}{6}}+2\fl{\frac{m-5-i}{3}}\\
&-\fl{\frac{(m-5-i)(m-6-i)}{6}}+\sum_{j=0}^{m-5-i}S(j+4)%\bigg(\fl{\frac{m-4-j-(i+1)}{3}}+1\bigg)
+ 2\sum_{j=0}^{m-6-i}S(j+5)%\bigg(\fl{\frac{m-5-j-(i+1)}{3}}+1\bigg)
.\end{align*}
Here we used that $i \le m-7$ and $2i>m-5$ imply that 
$$S(3)= \fl{\frac{m-4-i}{3}}+1\quad \text{and} \quad S(4)= \fl{\frac{m-5-i}{3}}+1.$$
Similarly one can see that
$$S(j+4)= \fl{\frac{m-5-j-i}{3}}+1\quad \text{and} \quad S(j+5)= \fl{\frac{m-6-j-i}{3}}+1.$$
Using this and Lemma \ref{floor_sums_cheat_sheet}, we obtain that
\begin{align*}
\Delta =&2\fl{\frac{m-4-i}{3}}-2\fl{\frac{(m-4-i)(m-5-i)}{6}}+2\fl{\frac{m-5-i}{3}}\\
&-\fl{\frac{(m-5-i)(m-6-i)}{6}}+\sum_{t=0}^{m-5-i}\fl{\frac{t}{3}}
+ 2\sum_{t=0}^{m-6-i}\fl{\frac{t}{3}}\\
=&2\fl{\frac{m-4-i}{3}}-2\sum_{t=0}^{m-4-i}\fl{\frac{t}{3}}+2\fl{\frac{m-5-i}{3}}+ 
2\sum_{t=0}^{m-6-i}\fl{\frac{t}{3}}\\
=&0.
\end{align*}
Hence $|\cup_{x=1,2,3} (G_{xa} \cup G_{xb})|=g(\mathcal{F}_m)$, just as claimed.

\bigskip
\noindent {\bf Case 3:} $m-7 \ge i$ and $2i = m-5$.

In this case we find $\min\{i-1,m-5-i\}=i-1=m-6-i$, 
$$S(3)= \fl{\frac{m-4-i}{3}}+2\quad \text{and} \quad S(4)= \fl{\frac{m-5-i}{3}}+1,$$
whence
$$|\cup_{x=1,2,3} (G_{xa} \cup G_{xb})|=g(\mathcal{F}_m)+\Delta,$$
with
\begin{align*}
\Delta =&-3m+15+3i+2\fl{\frac{m-4-i}{3}}-2\fl{\frac{(m-4-i)(m-5-i)}{6}}+2\fl{\frac{m-5-i}{3}}\\
&-\fl{\frac{(m-5-i)(m-6-i)}{6}}+\sum_{j=0}^{m-6-i}S(j+4)
+ 2\sum_{j=0}^{m-6-i}S(j+5)\\
=&2\fl{\frac{m-4-i}{3}}-2\fl{\frac{(m-4-i)(m-5-i)}{6}}+2\fl{\frac{m-5-i}{3}}\\
&-\fl{\frac{(m-5-i)(m-6-i)}{6}}+\sum_{t=1}^{m-5-i}\fl{\frac{t}{3}}
+ 2\sum_{t=0}^{m-6-i}\fl{\frac{t}{3}}\\
=&2\fl{\frac{m-4-i}{3}}-2\sum_{t=0}^{m-4-i}\fl{\frac{t}{3}}+2\fl{\frac{m-5-i}{3}}+ 
2\sum_{t=0}^{m-6-i}\fl{\frac{t}{3}}\\
=&0.
\end{align*}

\noindent {\bf Case 4:} $m-7 \ge i$ and $2i < m-5$.

We have 
$$|\cup_{x=1,2,3} (G_{xa} \cup G_{xb})|=g(\mathcal{F}_m)+\Delta,$$
where this time
\begin{align*}
\Delta =&-3m+12+3i+\fl{\frac{m-4-i}{3}}-2\fl{\frac{(m-4-i)(m-5-i)}{6}}+\fl{\frac{m-5-i}{3}}\\
&-\fl{\frac{(m-5-i)(m-6-i)}{6}}+S(3)+S(4)+\sum_{j=0}^{i-1}S(j+4)
+ 2\sum_{j=0}^{i-1}S(j+5)+S(i+5)\\
=& -\frac{(m-i-2)(m-i-3)}{2}+S(3)+S(4)+\sum_{j=0}^{i-1}S(j+4)
+ 2\sum_{j=0}^{i-1}S(j+5)+S(i+5)\\
=& -\frac{(m-i-2)(m-i-3)}{2}+S(3)+2S(4)+3\sum_{j=5}^{i+3}S(j)
+ 2S(i+4)+S(i+5).
\end{align*}
Here the second equality follows by explicit computation considering the three possible congruence classes of $m-4-i$ modulo three.

We will first show that 
\begin{equation}\label{eq:claimjtilde}
\sum_{j=3}^{i+3}S(j)=\fl{\frac{(m-i-1)(m-i-2)}{6}}.
\end{equation}
Let $\tilde{j}$ be the unique integer between $3$ and $i+3$ satisfying that $i+1$ divides $m-\tilde{j}$. To show Equation \eqref{eq:claimjtilde}, observe that 
$$\fl{\frac{m-j}{i+1}}=
\begin{cases}
\frac{m-\tilde{j}}{i+1}, & \text{if } 3 \le j \le \tilde{j},\\
\frac{m-\tilde{j}}{i+1}-1, & \text{if } \tilde{j} < j \le i+3.
\end{cases}$$
Hence
\begin{align*}
\sum_{j=3}^{i+3}S(j) & = \sum_{j=3}^{i+3}\sum_{\ell_2=1}^{\frac{m-\tilde{j}}{i+1}-1} \left(\fl{\frac{m-j-\ell_2(i+1)}{3}}+1\right)+\sum_{j=3}^{\tilde{j}}\left(\fl{\frac{m-j-(m-\tilde{j})}{3}}+1\right)\\
&=\sum_{\ell_2=1}^{\frac{m-\tilde{j}}{i+1}-1}\sum_{j=3}^{i+3} \left(\fl{\frac{m-\ell_2(i+1)-j}{3}}+1\right)+\sum_{j=3}^{\tilde{j}}\left(\fl{\frac{\tilde{j}-j}{3}}+1\right)\\
&=\sum_{\ell_2=1}^{\frac{m-\tilde{j}}{i+1}-1}\sum_{j=0}^{i} \fl{\frac{m-\ell_2(i+1)-j}{3}}+\sum_{j=0}^{\tilde{j}-3}\fl{\frac{\tilde{j}-j}{3}}\\
&=\sum_{\ell_2=1}^{\frac{m-\tilde{j}}{i+1}-1}\left(\sum_{t=0}^{m-\ell_2(i+1)} \fl{\frac{t}{3}}-\sum_{t=0}^{m-(\ell_2+1)(i+1)} \fl{\frac{t}{3}}\right)+\sum_{j=0}^{\tilde{j}-3}\fl{\frac{\tilde{j}-j}{3}}\\
&=
\sum_{t=0}^{m-(i+1)} \fl{\frac{t}{3}}-\sum_{t=0}^{m-\frac{m-\tilde{j}}{i+1}(i+1)} \fl{\frac{t}{3}}+\sum_{j=0}^{\tilde{j}-3}\fl{\frac{\tilde{j}-j}{3}}\\
&=
\sum_{t=0}^{m-(i+1)} \fl{\frac{t}{3}}-\sum_{t=0}^{\tilde{j}} \fl{\frac{t}{3}}+\sum_{j=0}^{\tilde{j}-3}\fl{\frac{\tilde{j}-j}{3}}=\sum_{t=0}^{m-(i+1)} \fl{\frac{t}{3}}.
\end{align*}
At this point Equation \eqref{eq:claimjtilde} follows directly by applying Lemma \ref{floor_sums_cheat_sheet}.

Next observe that 
$$S(i+4)=S(3)-\left(\fl{\frac{m-4-i}{3}}+1\right) \quad \text{and} \quad S(i+5)=S(4)-\left(\fl{\frac{m-5-i}{3}}+1\right).$$
Using this and Equation \eqref{eq:claimjtilde}, we conclude that
\begin{equation}\label{eq:claimjtilde2}
\sum_{j=4}^{i+4}S(j)=\fl{\frac{(m-i-1)(m-i-2)}{6}}-\fl{\frac{m-1-i}{3}},
\end{equation}
and
\begin{equation}\label{eq:claimjtilde3}
\sum_{j=5}^{i+5}S(j)=\fl{\frac{(m-i-1)(m-i-2)}{6}}-\fl{\frac{m-1-i}{3}}-\fl{\frac{m-2-i}{3}}.
\end{equation}
Equations \eqref{eq:claimjtilde}, \eqref{eq:claimjtilde2} and \eqref{eq:claimjtilde3} imply that
\begin{equation*}
\begin{split}
S(3)+2S(4)+3\sum_{j=5}^{i+3}S(j)
+ 2S(i+4)+S(i+5)=\\ 3\fl{\frac{(m-i-1)(m-i-2)}{6}}&-2\fl{\frac{m-1-i}{3}}-\fl{\frac{m-2-i}{3}}    
\end{split}
\end{equation*}
Treating the three congruence classes of $m-i$ modulo three separately, a direct computation then shows that
\begin{equation*}
S(3)+2S(4)+3\sum_{j=5}^{i+3}S(j)
+ 2S(i+4)+S(i+5)=\frac{(m-i-2)(m-i-3)}{2}.    
\end{equation*}
But then we can conclude that $\Delta=0$, just as we wanted.

%\textcolor{ForestGreen}{MARIE this is just a suggestion for readability: the first two lines (equivalent to $|G_1\cup G_2\cup G_3|-|G_{1a}\cup G_{2a}\cup G_{3a}|$) can be simplified a bit if we allow for some cases.}
%\begin{align*}
%    &\fl{\frac{m-4-i}{3}}-2\fl{\frac{(m-4-i)(m-5-i)}{6}}+\fl{\frac{m-5-i}{3}}-\fl{\frac{(m-5-i)(m-6-i)}{6}}\\
%    =&
%    \begin{cases}
%        \frac{m-4-i}{3}-2\frac{(m-4-i)(m-5-i)}{6}+\frac{m-4-i}{3}-1-\frac{(m-5-i)(m-6-i)-2}{6} & \text{ if } m-i-4\equiv 0 \mod3\\
%        \frac{m-5-i}{3}-2\frac{(m-4-i)(m-5-i)}{6}+\frac{m-5-i}{3}-\frac{(m-5-i)(m-6-i)}{6}& \text{ if } m-i-5\equiv 0 \mod3\\
%        \frac{m-6-i}{3}-2\frac{(m-4-i)(m-5-i)-2}{6}+\frac{m-6-i}{3}-{\frac{(m-5-i)(m-6-i)}{6}}& \text{ if } m-i-6\equiv 0 \mod3\\
%    \end{cases}
%\end{align*}
%\textcolor{ForestGreen}{in any case, this plus $-3m+12+3i$ is equal to $-\frac{(m-i-2)(m-i-3)}{2}$. 
%what is left is $s(3) + 2s(4) + 2s(i+4) + s(i+5) + 3 \sum_{j=5}^{i+3}  s(j)$ using the notation from the other file, which can be shown to be equal to $\frac{(m-i-2)(m-i-3)}{2}$, but maybe/probably/hopefully that argument can be shortened as well (and then the case where $2i=m-5$ )}

This completes the proof of Lemma \ref{counting-ismall} and thereby also the proof of Theorem \ref{conj:gaps}.

\section*{Acknowledgements}
This work has been supported by Villum Fonden under Grant VIL52303.

\end{document}